\documentclass[11pt]{article}
\usepackage{amsmath, amsthm,amssymb,amscd,verbatim}
\usepackage[mathscr]{eucal}
\usepackage{enumerate,multicol}
\usepackage{url}
\usepackage{float}
\usepackage{tikz}
\usepackage{bm}

\usepackage{wasysym}

\usepackage[margin=1in]{geometry}

\newtheorem{theorem}{Theorem}[section]
\newtheorem{lemma}[theorem]{Lemma}
\newtheorem{corollary}[theorem]{Corollary}
\newtheorem{proposition}[theorem]{Proposition}

\newtheorem{fact}[theorem]{Fact}

\theoremstyle{definition}
\newtheorem{definition}[theorem]{Definition}

\newtheorem{remark}[theorem]{Remark}
\newtheorem{question}[theorem]{Question}

\def\acl{\operatorname{acl}}
\def\dcl{\operatorname{dcl}}
\def\tp{\operatorname{tp}}

\def\Aut{\operatorname{Aut}}
\def\M{\mathbb M}
\def\cL{\mathcal L}
\def\Diag{\operatorname{Diag}}
\newcommand{\seq}{\subseteq}
\newcommand{\thrn}{\text{\thorn}}
\newcommand{\br}{\bm{r}}
\newcommand{\bs}{\bm{s}}
\newcommand{\bt}{\bm{t}}
\newcommand{\eq}{\mathrm{eq}}

\def\Ind{\setbox0=\hbox{$x$}\kern\wd0\hbox to 0pt{\hss$\mid$\hss}
\lower.9\ht0\hbox to 0pt{\hss$\smile$\hss}\kern\wd0}

\def\Notind{\setbox0=\hbox{$x$}\kern\wd0\hbox to 0pt{\mathchardef
\nn=12854\hss$\nn$\kern1.4\wd0\hss}\hbox to
0pt{\hss$\mid$\hss}\lower.9\ht0 \hbox to 0pt{\hss$\smile$\hss}\kern\wd0}

\def\ind{\mathop{\mathpalette\Ind{}}}
\def\nind{\mathop{\mathpalette\Notind{}}}

\newcommand{\gind}{\ind}
\newcommand{\ngind}{\nind}
\newcommand{\iind}{\ter{I}}
\newcommand{\niind}{\nter{I}}

\newcommand{\ter}[1]{\ind^{\!\!{#1}}}
\newcommand{\nter}[1]{\nind^{\!\!{#1}}}
\newcommand{\trt}[1]{\ind^{\!\!\textnormal{#1}}}
\newcommand{\ntrt}[1]{\nind^{\!\!\textnormal{#1}}}
\AtBeginDocument{%
   \def\MR#1{}
}

\title{Independence in generic incidence structures}
\author{Gabriel Conant\\
\small{University of Notre Dame}
 \and 
 Alex Kruckman\\
 \small{Indiana University Bloomington}}
\date{April 2, 2018}

\begin{document}

\maketitle
\abstract{We study the theory $T_{m,n}$ of existentially closed incidence structures omitting the complete incidence structure $K_{m,n}$, which can also be viewed as existentially closed $K_{m,n}$-free bipartite graphs. In the case $m = n = 2$, this is the theory of existentially closed projective planes. We give an $\forall\exists$-axiomatization of $T_{m,n}$, show that $T_{m,n}$ does not have a countable saturated model when $m,n\geq 2$, and show that the existence of a prime model for $T_{2,2}$ is equivalent to a longstanding open question about finite projective planes. Finally, we analyze model theoretic notions of complexity for $T_{m,n}$. We show that $T_{m,n}$ is NSOP$_1$, but not simple when $m,n\geq 2$, and we show that $T_{m,n}$ has weak elimination of imaginaries but not full elimination of imaginaries. These results rely on combinatorial characterizations of various notions of independence, including algebraic independence, Kim independence, and forking independence.}

\section{Introduction}

Many examples in model theory can be obtained using the following recipe: Start with a universally axiomatized base theory $T$ with class of models $\mathcal{K}$, and consider the subclass $\mathcal{K}^*$ of existentially closed models. If $\mathcal{K}^*$ is an elementary class, axiomatized by a theory $T^*$, then $T^*$ is the model companion of $T$, which we often call the generic theory of $\mathcal{K}$. The large saturated models of $T^*$ can be viewed as universal domains for $\mathcal{K}$. Being model complete, $T^*$ eliminates quantifiers at least down to existential formulas, and the completions of $T^*$ tend to be model theoretically tamer than arbitrary completions of $T$. 

We apply this recipe to incidence structures. An incidence structure is a set $P$ of points and a set $L$ of lines, together with a binary relation $I$, called incidence, between the points and the lines. In this paper, we focus on the theory $T^p_{m,n}$ of incidence structures which omit the complete incidence structure $K_{m,n}$ (consisting of $m$ points incident to $n$ lines). Equivalently, these are incidence structures in which every $m$ points simultaneously lie on at most $n-1$ lines and every $n$ lines intersect in at most $m-1$ points. We show that this theory has a model companion $T_{m,n}$ and study its model theoretic properties. In Section~\ref{sec:props}, we give a $\forall\exists$-axiomatization of $T_{m,n}$ (Theorem~\ref{thm:modelcompanion}), characterize algebraic closure (Proposition~\ref{prop:Iacl}), and prove an ``almost quantifier elimination'' result (Proposition~\ref{prop:partialqe}). In Section~\ref{sec:countable}, we consider countable saturated and prime models of $T_{m,n}$. In Section~\ref{sec:NSOP1}, we position $T_{m,n}$ in the classification theory hierarchy,  characterize various notions of independence, and prove weak elimination of imaginaries. This work fits into the following contexts. 

First, in the special case when $m = n = 2$, the models of $T_{m,n}$ are existentially closed combinatorial projective planes, i.e., existentially closed incidence structures in which every pair of points determine a unique line and every pair of lines intersect in a unique point. The class of existentially closed projective planes was previously examined by Kegel~\cite{Kegel}, who proposed a further study of the model theory of this class, but it appears that this program was not pursued further until now. Later, Baldwin~\cite{baldwin} and Tent and Zilber~\cite{TZ} used the Hrushovski method to produce examples of projective planes with surprising properties. These planes are generic for certain restricted classes of projective planes, defined using predimension functions.

The theory $T_{m,n}$ is also interesting from the perspective of the model theory of graphs. The class of graphs has a model companion, the unique countable model of which is known as the ``random graph". This structure has a bipartite counterpart, the ``random bipartite graph". When it is viewed as a structure in a language with unary predicates for the bipartition, the theory of the random bipartite graph is the model companion of the theory of arbitrary incidence structures. Both the random graph and the random bipartite graph have well-understood theories: they are $\aleph_0$-categorical, simple, unstable, and have quantifier elimination and trivial algebraic closure. 

In the same way, $T_{m,n}$ can be viewed as the theory of existentially closed bipartite graphs omitting the complete bipartite graph $K_{m,n}$. Thus the theories $T_{m,n}$ are bipartite counterparts to the theories $T_n$ of generic $K_n$-free graphs (existentially closed graphs omitting a complete subgraph of size $n$). The theories $T_n$, introduced by Henson~\cite{Henson}, are important examples of countably categorical theories exhibiting the properties TP$_2$, SOP$_3$, and NSOP$_4$ (see~\cite{Sh500}, and also~\cite{conant} for discussion of the model theoretic properties of Henson graphs). We show that for $m,n\geq 2$, $T_{m,n}$ also has TP$_2$, but it is NSOP$_1$, hence tamer, in a sense, than the Henson graphs. 

In another contrast to the Henson graphs, $T_{m,n}$ is not countably categorical when $m,n\geq 2$. In fact, we show that in this case, $T_{m,n}$ has continuum-many types over the empty set, and hence has no countable saturated model (Theorem \ref{thm:small}). The question of whether $T_{m,n}$ has a prime model seems to be a hard combinatorial problem; in the case $m=n=2$, we show that it is equivalent to a longstanding open problem in the theory of projective planes (Theorem \ref{thm:primemodel}).

Finally, from the point of view of classification theory, the generic theory recipe has been a fruitful source of examples of simple theories (see~\cite{ChaPil}, for example). 
Recently, many examples of generic theories which are not simple have been shown to be NSOP$_1$ (see \cite{ArtemNick}, \cite{Kr}, \cite{KrRa}). The theories $T_{m,n}$ are further examples of this phenomenon, and they provide good combinatorial examples of properly NSOP$_1$ theories. One noteworthy example, parallel to our work, is that of \emph{Steiner triple systems}: incidence structures in which every line contains exactly three points and any two points are contained in a unique line. Barbina and Casanovas \cite{BaCa} have recently shown that in the language of quasi-groups, the class of finite Steiner triple systems is a Fra\"{i}ss\'{e} class, whose Fra\"{i}ss\'{e} limit is properly NSOP$_1$.

The main tool for proving that generic theories are NSOP$_1$ is the independence relation criterion developed by Chernikov and Ramsey~\cite{ArtemNick} and refined by Kaplan and Ramsey~\cite{KRKim}: If there is a notion of independence between subsets of the monster model of a complete theory, which satisfies certain axioms, then that theory must be NSOP$_1$, and the independence relation must be Kim independence. Applying this criterion to $T_{m,n}$, we obtain a characterization of Kim independence in terms of incidence-free disjoint amalgamation of algebraically closed sets (Definition \ref{def:KI}, Theorem \ref{thm:NSOP1}). In addition, we characterize other notions of independence in $T_{m,n}$, including forking independence, thorn independence, algebraic independence, and a stationary independence relation which is stronger than all of these. While Kim independence is clearly the most well-behaved independence relation in any NSOP$_1$ theory, in examples we often find multiple interesting notions of independence, and the relationships between these notions are likely to be of importance in the general theory. For example, our analysis of Kim independence in $T_{m,n}$ exposes a rather intricate relationship to algebraic independence (Lemma \ref{lem:SFCW}, Proposition \ref{prop:Kequiv}). This relationship is a key tool in the characterization of dividing in $T_{m,n}$ (Corollary~\ref{cor:div}), as well as in the proof of weak elimination of imaginaries (Theorem \ref{thm:WEI}).

\section{Axiomatization and almost quantifier elimination}\label{sec:props}

We consider incidence structures in the language $\cL=\{P,L,I\}$, where $P$ and $L$ are unary predicates for ``points" and ``lines", and $I$ is a binary relation for ``incidence" between points and lines. An \emph{incidence structure} is an $\cL$-structure such that the interpretations of $P$ and $L$ partition the domain, and incidences $I(x,y)$ only hold on pairs $(x,y)\in P\times L$. 

Given natural numbers $m$ and $n$, let $K_{m,n}$ be the the incidence structure consisting of $m$ points and $n$ lines, such that every point is incident to every line. We say that an incidence structure is \emph{$K_{m,n}$-free} if it does not contain a substructure isomorphic to $K_{m,n}$. In other words, an incidence structure is $K_{m,n}$-free if every $m$ points simultaneously lie on at most $n-1$ lines, and every $n$ lines intersect in at most $m-1$ points. While the use of the word ``line" is somewhat unnatural when $m,n\neq 2$, we nevertheless find this geometric language useful for its evocative power.

An equivalent viewpoint is that incidence structures are bipartite graphs, where $P$ and $L$ are the two pieces of the bipartition and $I$ is the edge relation between them. In this interpretation, $K_{m,n}$ is simply the complete bipartite graph with $m+n$ vertices partitioned into two pieces of sizes $m$ and $n$. However, it is worth emphasizing that we distinguish the two pieces of the partition, calling one part ``points" and one part ``lines". Thus,  if $m\neq n$, a ``$K_{m,n}$-free bipartite graph" is not necessarily $K_{n,m}$-free. This convention will be convenient for several combinatorial constructions in this paper, as well as for model theoretic reasons later on (e.g.\ a single point and a single line have different types over $\emptyset$). 

For the rest of this section, we fix $m,n\geq 1$. 

\begin{definition}
Let $T^p_{m,n}$ denote the universal $\cL$-theory of $K_{m,n}$-free incidence structures. Let $T^c_{m,n}$ denote the theory of $K_{m,n}$-free incidence structures such that every $m$ points simultaneously lie on exactly $n-1$ lines, and every $n$ lines intersect in exactly $m-1$ points.  The letters $p$ and $c$ stand for ``partial'' and ``complete''.
 \end{definition}

\begin{remark}
$T^c_{2,2}$ can be classically recognized as the theory of \emph{combinatorial projective planes} in which every two points are incident to a unique line, and every two lines are incident to a unique point. One caveat is that the definition of a projective plane usually involves a non-degeneracy axiom: there are four distinct points no three of which lie on a common line. We will soon consider the model companion of $T^c_{2,2}$, which will automatically include this axiom. 
\end{remark}

\begin{proposition}\label{prop:free-closure}
Every model of $T^p_{m,n}$ embeds in a model of $T^c_{m,n}$.
\end{proposition}
\begin{proof}
Let $M\models T^p_{m,n}$. For all $k\in \omega$, we construct a model $F_k(M)\models T^p_{m,n}$, such that $M\seq F_k(M)\seq F_{k+1}(M)$. Let $F_0(M)=M$.

Fix $k\in\omega$ and suppose $F_k(M)\models T^p_{m,n}$ has been constructed. Let $\Sigma^P_k(M)$ be the collection of all $m$-element sets of points in $F_k(M)$ which simultaneously lie on at most $n-2$ lines in $F_k(M)$, and let $\Sigma^L_k(M)$ be the collection of all $n$-element sets of lines in $F_k(M)$ which intersect in at most $m-2$ points in $F_k(M)$. Let $F_{k+1}(M)$ be obtained by adding to $F_k(M)$:
\begin{enumerate}[$(i)$]
\itemsep0em 
\parskip0em
\item pairwise distinct lines $y_\sigma$, for each $\sigma\in \Sigma^P_k(M)$, with new incidences $I(x,y_\sigma)$ for all $x\in\sigma$, and
\item pairwise distinct points $x_\sigma$, for each $\sigma\in\Sigma^L_k(M)$, with new incidences $I(x_\sigma,y)$ for all $y\in \sigma$.
\end{enumerate}
We claim $F_{k+1}(M)\models T^p_{m,n}$. To see this, suppose, for a contradiction, that  $F_{k+1}(M)$ contains a copy of $K_{m,n}$. Since $F_k(M)$ is $K_{m,n}$-free, this copy of $K_{m,n}$ must contain a line $y_\sigma$ for some $\sigma\in\Sigma^P_k(M)$ or a point $x_\sigma$ for some $\sigma\in \Sigma^L_k(M)$. Suppose $K_{m,n}$ contains such a $y_\sigma$ (the other case is symmetric). The points in $F_{k+1}(M)$ incident to $y_\sigma$ are precisely the $m$ points in $\sigma$, so these must be the $m$ points in the copy of $K_{m,n}$. By assumption, the points in $\sigma$ simultaneously lie on at most $n-2$ lines in $M$, and so there must be some $y_\tau$ in the copy of $K_{m,n}$ for some $\tau\neq\sigma$ in $\Sigma^P_k(M)$. But $y_\tau$ is only incident to the $m$ points in $\tau$, and so $\sigma\neq\tau$ implies that $y_\tau$ is not incident to some point in $\sigma$, which is a contradiction.

This finishes the construction of $F_k(M)$, for $k\in\omega$. Let $F(M)=\bigcup_{k\in\omega}F_k(M)$, and note that $M$ embeds in $F(M)$. By construction, $F(M)$ is $K_{m,n}$-free, and so $F(M)\models T^p_{m,n}$. We claim $F(M)\models T^c_{m,n}$. Indeed, suppose for a contradiction that there are distinct points $x_1,\ldots,x_m$ in $F(M)$ which simultaneously lie on at most $n-2$ lines in $F(M)$. By construction, there is some $k\in\omega$ such that $x_1,\ldots,x_m\in F_k(M)$, and every line containing $x_1,\ldots,x_m$ is in $F_k(M)$. Then, in the construction of $F_{k+1}(M)$, we added a new line containing $x_1,\ldots,x_m$, which is a contradiction. By a symmetric argument, every $n$ lines in $F(M)$ intersect in exactly $m-1$ points in $F(M)$. 
\end{proof}

The extension of a model of $T^p_{m,n}$ to a model of $T^c_{m,n}$ is not unique in general (see Lemma~\ref{lem:notfree} below). The particular construction used in the previous proof was introduced by Hall~\cite{Hall} (in the case of projective planes, $m=n=2$). This construction will also be useful later on, and so we give it a name.

\begin{definition}
Given $M\models T^p_{m,n}$, let $F(M)=\bigcup_{k\in\omega}F_k(M)$ as in the proof of Proposition \ref{prop:free-closure}. $F(M)$ is called the \emph{free completion} of $M$.
\end{definition}

\begin{corollary}\label{cor:free-closure}
If $M\models T^p_{m,n}$ then $F(M)\models T^c_{m,n}$.
\end{corollary}

We now turn to the generic theory of $K_{m,n}$-free incidence structures. 

\begin{definition}\label{def:diagram}
Let $A$ be a model of $T^p_{m,n}$, with elements bijectively labeled by the variables $\overline{z}$. The \emph{diagram} $\Diag_A(\overline{z})$ is the set of all atomic and negated atomic formulas true in $A$. 
\end{definition}

Note that $\Diag_A(\overline{z})$ does not contain any positive instances of equality between distinct variables. If $A$ is finite, we identify $\Diag_A(\overline{z})$ with the formula $\bigwedge_{\varphi\in \Diag_A(\overline{z})} \varphi$.

\begin{definition}\label{def:safe}
Let $\Delta(\overline{x};\overline{y})$ be a set of atomic and negated atomic formulas in the variables $\overline{x}$ and $\overline{y}$. $\Delta$ is a \emph{safe diagram} if:
\begin{enumerate}
\itemsep0em 
\parskip0em
\item There is some finite $A\models T^p_{m,n}$ and a labeling of the elements of $A$ by $\overline{x}\overline{y}$ such that $\Delta(\overline{x};\overline{y}) = \Diag_A(\overline{x};\overline{y})$. 
\item For each $y_i$ in $\overline{y}$, with $P(y_i)\in\Delta$, there are at most $n-1$ variables $x_j$ in $\overline{x}$ such that $L(x_j),I(y_i,x_j)\in \Delta$. 
\item For each $y_i$ in $\overline{y}$, with $L(y_i)\in\Delta$, there are at most $m-1$ variables $x_j$ in $\overline{x}$ such that $P(x_j),I(x_j,y_i)\in \Delta$. 
\end{enumerate}
\end{definition}

Given a safe diagram $\Delta(\overline{x};\overline{y})$, let $\widehat{\Delta}(\overline{x})$ be the subset of $\Delta$ containing those atomic formulas which only mention variables in $\overline{x}$. Let $T_{m,n}$ be the theory 
\[
T^c_{m,n} \cup\left\{\forall \overline{x}\, \left(\widehat{\Delta}(\overline{x}) \rightarrow \exists\overline{y}\,\Delta(\overline{x};\overline{y})\right)\mid \Delta \text{ is a safe diagram}\right\}.
\]

\begin{theorem}\label{thm:modelcompanion}
$T_{m,n}$ is the model companion of $T^p_{m,n}$ and of $T^c_{m,n}$.
\end{theorem}
\begin{proof}
We show that the models of $T_{m,n}$ are exactly the existentially closed models of $T^p_{m,n}$. So let $M\models T^p_{m,n}$ be existentially closed. 
To show first that $M\models T^c_{m,n}$, suppose for contradiction that $a_1,\ldots,a_m$ are distinct points in $M$, which simultaneously lie on at most $n-2$ distinct lines in $M$. By Proposition~\ref{prop:free-closure}, $M$ embeds in a model $F(M)$ of $T^c_{m,n}$, and $F(M)$ satisfies the existential formula asserting that $a_1,\dots,a_m$ simultaneously lie on $n-1$ distinct lines. This contradicts the assumption that $M$ is existentially closed. A symmetric argument shows that every $n$ distinct lines in $M$ intersect in $n-1$ distinct points.

To show that $M\models T_{m,n}$, let $\Delta(\overline{x};\overline{y})$ be a safe diagram, witnessed by the finite $A\models T^p_{m,n}$. Let $\overline{a}$ be a tuple from $M$ satisfying $\widehat{\Delta}(\overline{x})$. Build a structure $N$ containing $M$ by adding new elements $\overline{b}$ corresponding to the variables $\overline{y}$, making them points or lines as specified by $\Delta$, and adding only the new incidences between $\overline{b}$ and $\overline{a}$ specified by $\Delta$. Then $N\models T^p_{m,n}$. Indeed, suppose $K\seq N$ is a copy of $K_{m,n}$. Since $M$ is $K_{m,n}$-free, $K$ must contain an element of $\overline{b}$, say $b_i$. We assume $b_i$ is a point in $N$ (a symmetric argument works when $b_i$ is a line). Since $b_i$ lies on at most $n-1$ lines in $\overline{a}$, and thus at most $n-1$ lines in $M$, $K$ must contain another element $b_j \in \overline{b}$ such that $b_j$ is a line containing $b_i$. So $K$ contains both a point and a line in $\overline{b}$. Since the elements of $\overline{b}$ are not incident with any elements of $N$ outside of $\overline{ab}$, it follows that $K$ is entirely contained in $\overline{ab}$, contradicting the fact that $A\models T^p_{m,n}$. So we have $N\models T^p_{m,n}$, and $N\models \exists\overline{y}\,\Delta(\overline{a},\overline{y})$. Since $M$ is existentially closed, it already contains a tuple $\overline{b}'$ satisfying $\Delta(\overline{a},\overline{b}')$, and hence $M$ satisfies the axiom $\forall \overline{x}\,(\widehat{\Delta}(\overline{x})\rightarrow \exists\overline{y}\,\Delta(\overline{x};\overline{y}))$. 

We have shown that the existentially closed models of $T^p_{m,n}$ satisfy $T_{m,n}$. Conversely, suppose $M\models T_{m,n}$, and assume for contradiction that $M$ is not existentially closed. Then there exist $M\seq N\models T^p_{m,n}$, tuples $\overline{a}$ from $M$ and $\overline{b}$ from $N$, and a quantifier-free formula $\varphi(\overline{x};\overline{y})$, such that $N\models \varphi(\overline{a};\overline{b})$, but $M\models \lnot \exists \overline{y}\,\varphi(\overline{a};\overline{y})$. Note that we may assume that each element of $\overline{b}$ is in $N\backslash M$ and that the elements of $\overline{a}$ and $\overline{b}$ are distinct. Further, letting $A$ be the substructure of $N$ with domain $\overline{a}\overline{b}$, we may assume that $\varphi(\overline{x};\overline{y}) = \Diag_A(\overline{x};\overline{y})$. 

Since $M\models T^c_{m,n}$, any $n$ distinct lines in $M$ already intersect in exactly $m-1$ points in $M$. Since $N\models T^p_{m,n}$, it follows that each point in $\overline{b}$ lies on at most $n-1$ lines in $\overline{a}$. Symmetrically, each line in $\overline{b}$ contains at most $m-1$ points in $\overline{a}$. Hence $\varphi$ is a safe diagram, and since $M\models \widehat{\varphi}(\overline{a})$, $M$ fails to satisfy the corresponding axiom of $T_{m,n}$.

We have established that $T_{m,n}$ is the model companion of $T^p_{m,n}$. The fact that $T_{m,n}$ is also the model companion of $T^c_{m,n}$ follows from the fact that $T^p_{m,n}$ and $T^c_{m,n}$ are companions: every model of $T^c_{m,n}$ is a model of $T^p_{m,n}$, and every model of $T^p_{m,n}$ embeds in a model of $T^c_{m,n}$ by Proposition \ref{prop:free-closure}. 
\end{proof}

\begin{proposition}\label{prop:amalgamation}
Models of $T^c_{m,n}$ satisfy the disjoint amalgamation property.
\end{proposition}
\begin{proof}
Fix models $A$, $B$, and $C$ of $T^c_{m,n}$ and embeddings $f\colon A\to B$ and $g\colon A\to C$. Let $D$ be free amalgamation of $B$ and $C$ over $A$, i.e., take the domain of $D$ to be the disjoint union of $B$ and $C$ over $A$ (with $f(A)\seq B$ and $g(A)\seq C$ identified), and add only the incidences coming from the structures $B$ and $C$. 

Now $D\models T^p_{m,n}$. Indeed, any copy of $K_{m,n}$ in $D$ must include an element $b\in B\backslash A$ and an element $c\in C\backslash A$, since $B$ and $C$ are each $K_{m,n}$-free. Since there are no incidences between $B\backslash A$ and $C\backslash A$, it follows that $b$ and $c$ are either both points in $D$ or both lines in $D$. Assume $b$ and $c$ are both points in $D$ (the other case is symmetric). Again, since there are no incidences between $B\backslash A$ and $C\backslash A$, all $n$ lines $y_1,\ldots,y_n$ in the copy of $K_{m,n}$ must lie in $A$. But since $A\models T^c_{m,n}$, $y_1,\ldots,y_n$ intersect in exactly $m-1$ points in $A$, which then implies that $y_1,\ldots,y_n$ intersect in at least $m$ points in $B$ (since $b\in B\backslash A$). This contradicts $B\models T^p_{m,n}$. 

By Proposition \ref{prop:free-closure}, we can embed $D$ into $D'\models T^c_{m,n}$, and the images of $B$ and $C$ in $D'$ are disjoint over the image of $A$.
\end{proof}

\begin{corollary}
$T_{m,n}$ is the model completion of $T^c_{m,n}$, and $T_{m,n}$ is complete.
\end{corollary}
\begin{proof}
We already know that $T_{m,n}$ is the model companion of $T^c_{m,n}$ (Theorem~\ref{thm:modelcompanion}), and $T^c_{m,n}$ has the amalgamation property (Proposition~\ref{prop:amalgamation}), so $T_{m,n}$ is the model completion of $T^c_{m,n}$ (\cite[Proposition 3.5.18]{CK}). In the proof of Proposition \ref{prop:amalgamation}, $A$ can be empty, and so $T^c_{m,n}$ also has the (dis)joint embedding property. It follows that $T_{m,n}$ also has the joint embedding property, and hence $T_{m,n}$ is complete (\cite[Proposition 3.5.11]{CK}). 
\end{proof}

%\begin{proof}
%We already know that $T_{m,n}$ is the model companion of $T^c_{m,n}$ (Theorem~\ref{thm:modelcompanion}), and $T^c_{m,n}$ has the amalgamation property (Proposition~\ref{prop:amalgamation}), so $T_{m,n}$ is the model completion of $T^c_{m,n}$. In the proof of Proposition \ref{prop:amalgamation}, $A$ can be empty, and so $T^c_{m,n}$ also has the (dis)joint embedding property. It follows that $T_{m,n}$ is complete. 
%\end{proof}

Let $\M$ be a monster model of $T_{m,n}$. We call an incidence structure ``small" if its cardinality is smaller than the saturation of $\M$. We write $A\subset\M$ to mean $A$ is a small subset of $\M$.

\begin{proposition}\label{prop:extend}
Suppose $A\models T^c_{m,n}$ and $B\models T^p_{m,n}$ are small, and let $f\colon A\to \M$ and $g:A\to B$ be $\cL$-embeddings. Then there is an $\cL$-embedding $h\colon B\to \M$ such that $f=hg$. 
\end{proposition}
\begin{proof}
Without loss of generality assume $f$ is inclusion, and so $A\subset\M$. Let $\overline{a}$ be an enumeration of $A$ and $\overline{b}$ an enumeration of $B\backslash A$. Let $\Delta(\overline{a};\overline{b})$ be $\Diag_B(\overline{a};\overline{b})$. We want to realize $\Delta(\overline{a};\overline{y})$ in $\M$. By saturation and compactness, we may assume $\overline{a}$ and $\overline{b}$ are finite tuples from $A$ and $B\backslash A$, respectively. Since $A\models T^c_{m,n}$,  $\Delta(\overline{x};\overline{y})$ is a safe diagram (witnessed by the substructure of $B$ with domain enumerated by $\overline{ab}$). By assumption, $\M\models\widehat{\Delta}(\overline{a})$. So the desired realization of $\Delta(\overline{a};\overline{y})$ exists in $\M$, since $\M\models T_{m,n}$.
\end{proof}

\begin{corollary}\label{cor:extend}
Suppose $A\subset\M$ is a model of $T^c_{m,n}$ and $C\seq A$. Fix $\overline{a}\in A$ and suppose $f\colon A\to \M$ is an $\cL$-embedding fixing $C$ pointwise. Then $\overline{a}\equiv_C f(\overline{a})$.
\end{corollary}
\begin{proof}
Pick a small elementary submodel $M\preceq \M$ with $A\seq M$. By Proposition \ref{prop:extend}, we may extend $f$ to an $\cL$-embedding $h\colon M\to \M$. Since $T_{m,n}$ is model complete, $h$ is an elementary embedding, so $\tp_{\M}(\overline{a}/C) = \tp_M(\overline{a}/C) = \tp_{\M}(h(\overline{a})/h(C)) = \tp_{\M}(f(\overline{a})/C)$.
\end{proof}

\begin{definition}
Fix $B\models T^p_{m,n}$. We say that a subset $A\seq B$ is \emph{$I$-closed (in $B$)} if, for all pairwise distinct $a_1,\ldots,a_m\in A\cap P(B)$, if some $b\in L(B)$ is incident to each of the $a_i$, then $b\in A$, and, dually, for all pairwise distinct $b_1,\ldots,b_n\in A\cap L(B)$, if some $a\in P(B)$ is incident to each of the $b_j$, then $a\in A$. The \emph{$I$-closure in $B$} of a set $A\seq B$ is the smallest $I$-closed subset of $B$ containing $A$. If the $I$-closure of $A$ in $B$ is all of $B$, then we say $A$ \emph{generates} $B$.
\end{definition}

\begin{proposition}\label{prop:Iacl}
If $A\subset \M$ then $\acl(A)$ is the $I$-closure of $A$ in $\M$. Moreover, $\acl(A)$ is the smallest model $M$ of $T^c_{m,n}$ such that $A\seq M\seq \M$. 
\end{proposition}
\begin{proof}
For any $A\subset \M$, the $I$-closure of $A$ in $\M$ is contained in $\acl(A)$, contained in every model of $T^c_{m,n}$ containing $A$, and is itself a model of $T^c_{m,n}$. So it suffices to fix an $I$-closed set $A\subset \M$, and show $\acl(A)=A$. Suppose for contradiction that there is some $b_1\in \acl(A)\backslash A$, and that $\tp(b_1/A)$ has $k\in\omega$ realizations $b_1,b_2,\dots,b_k$, all of which are in $\acl(A)\backslash A$. Now $A$ and $\acl(A)$ are $I$-closed, so they are both models of $T^c_{m,n}$. By Proposition~\ref{prop:amalgamation}, we can find $D\models T^c_{m,n}$ and embeddings $f,g\colon \acl(A)\to D$ such that $f(\acl(A))$ and $g(\acl(A))$ are disjoint over $f(A) = g(A)$ (and we may assume $A\seq D$ and $f$ and $g$ restrict to inclusion on $A$). Then $A\models T^c_{m,n}$ is a substructure of $D\models T^c_{m,n}$ and $\M$, and so, by Proposition \ref{prop:extend}, there is an $\cL$-embedding $h\colon D\to\M$ which fixes $A$ pointwise. But then $h\circ f$ and $h\circ g$ are both $\cL$-embeddings $\acl(A)\to \M$ which fix $A$ pointwise, so the $2k$ elements $\{h(f(b_i)),h(g(b_i))\mid 1\leq i\leq k\}$ all satisfy $\tp(b_1/A)$ in $\M$ by Corollary \ref{cor:extend}, which is a contradiction.
\end{proof}

Putting together Corollary~\ref{cor:extend} and Proposition~\ref{prop:Iacl} yields the following result:

\begin{corollary}\label{cor:acltype}
Given $C\subset\M$ and tuples $\overline{a},\overline{a}'\in\M$, $\tp(\overline{a}/C) = \tp(\overline{a}'/C)$ if and only if $\acl(\overline{a}C) \cong \acl(\overline{a}'C)$ by an $\mathcal{L}$-isomorphism fixing $C$ pointwise and sending $\overline{a}$ to $\overline{a'}$.
\end{corollary}

$T_{m,n}$ does not have quantifier elimination, since quantifiers are necessary to describe the isomorphism type of the algebraic closure of a tuple. But we do get quantifier elimination down to existential quantification over the algebraic closure. This is frequently called \emph{almost quantifier elimination} in the literature (e.g.\ \cite{ChaPil}). 

\begin{definition}
Given a complete quantifier-free type $p$ in variables $x_1,\ldots,x_k$, let  $\Delta_p(\overline{x})$ denote the conjunction of all formulas in $p$. A \emph{basic existential formula} is one of the form
\[
\exists\overline{y}\, \Delta_p(\overline{x};\overline{y}),
\]
where $p$ is a complete quantifier-free type in variables $\overline{xy}$, such that $p$ implies that $\overline{x}$ generates $\overline{xy}$. In particular, if the tuple $\overline{y}$ is empty, $\Delta_p(\overline{x})$ is a basic existential formula. 
\end{definition}

The notion of a complete quantifier-free type differs from the notion of the diagram of a finite structure in that a complete quantifier-free type may contain positive instances of equality between the variables.

\begin{proposition}\label{prop:partialqe}
Modulo $T_{m,n}$, every formula is equivalent to a finite disjunction of basic existential formulas.
\end{proposition}
\begin{proof}
Fix $k>0$ and variables $x_1,\dots,x_k$. For each basic existential formula $\varphi(\overline{x})$, let $U_\varphi = \{p\in S_k(T_{m,n})\mid \varphi(\overline{x}) \in p\}$. Let $\mathcal{B} = \{U_\varphi\mid \varphi(\overline{x}) \text{ basic existential}\}$. It is straightforward to check that $\mathcal{B}$ is a basis for a topology $\mathcal{T}$ on $S_k(T_{m,n})$. Note also that $\mathcal{T}$ is compact, since it is weaker than the usual topology on $S_k(T_{m,n})$. In particular, any $\mathcal{T}$-clopen subset of $S_k(T_{m,n})$ is equal to a finite union of open sets in $\mathcal{B}$. Therefore, to prove the result, it suffices to show that $\mathcal{T}$ is in fact the usual topology on $S_k(T_{m,n})$. Since two comparable compact Hausdorff topologies on the same set are equal, we only need to show $\mathcal{T}$ is Hausdorff. 

Fix $k$-tuples $\overline{a}$ and $\overline{a}'$ in $\M$ such that $\tp(\overline{a}/\emptyset) \neq \tp(\overline{a}'/\emptyset)$. Note that we can build up $\acl(\overline{a})$ in stages: Let $A_0$ be the induced substructure of $\M$ on $\overline{a}$, and given $A_t$, let $A_{t+1}$ be $A_t$ together with those lines in $\M$ which are incident with $m$ distinct points in $A_t$ and those points in $\M$ which are incident with $n$ distinct lines in $A_t$. Then $A_t$ is finite for all $t$, and $\acl(\overline{a}) = \bigcup_{t\in \omega} A_t$. Similarly, we can build up $\acl(\overline{a}')$ in stages $(A_t')_{t\in\omega}$.  We claim that, for some $t\in\omega$, $A_t$ and $A'_t$ are not isomorphic via an isomorphism mapping $\overline{a}$ to $\overline{a}'$. Indeed, if this were not the case then a direct application of K\"{o}nig's lemma would produce an isomorphism from $\acl(\overline{a})$ to $\acl(\overline{a}')$, which maps $\overline{a}$ to $\overline{a}'$, contradicting Corollary~\ref{cor:acltype}. So let $t$ be as claimed. Let $\overline{b}$ enumerate $A_t\backslash A_0$, let $q$ be the quantifier-free type of $\overline{a}\overline{b}$, and let $\varphi(\overline{x})$ be the basic existential formula $\exists\overline{y}\,\Delta_q(\overline{x};\overline{y})$. Similarly, let $\varphi'(\overline{x})$ be $\exists\overline{y}'\,\Delta_{q'}(\overline{x};\overline{y}')$, where $q'$ is the quantifier-free type of $\overline{a}'\overline{b}'$ and $\overline{b}'$ enumerates $A_t'\backslash A_0'$. Since $\varphi$ and $\varphi'$ assert that the structures generated from $\overline{x}$ in $k$ stages are isomorphic to $A_t$ and $A_t'$, respectively, the formula $\varphi(\overline{x})\land \varphi'(\overline{x})$ is inconsistent. So $U_{\varphi}$ and $U_{\varphi'}$ are disjoint $\mathcal{T}$-open sets separating $\tp(\overline{a}/\emptyset)$ and $\tp(\overline{a}'/\emptyset)$. 
\end{proof}

We close this section with some remarks on the special case of $T_{2,2}$ and combinatorial projective planes. In $T_{2,2}$, algebraic closure coincides with definable closure. In particular, we have a definable function $H\colon \M\times \M\to \M$ such that, for distinct points (resp. lines) $x$ and $y$, $H(x,y)=H(y,x)$ is the unique line (resp. point) incident to both $x$ and $y$, and $H(x,y)=x$ otherwise.  If we expand the language by a symbol for $H$, and let 
$\widehat{T}_{2,2}$ be $T_{2,2}$ together with the definition of $H$, then $\widehat{T}_{2,2}$ has quantifier elimination.

Models of $T_{2,2}$ are non-Desarguesian projective planes, i.e., they fail to satisfy Desargues's theorem (see~\cite{Hall} for the definition). If we restrict our attention to the Desarguesian projective planes, we find that this class does not have a model companion. This follows from the fact that the theory of Desarguesian planes with constants naming a quadrangle (four points with no three collinear) is bi-interpretable with the theory of division rings. Moreover, the interpretations are especially well-behaved (in particular, the formulas involved are existential), and one can check that the existentially closed Desarguesian planes correspond exactly to the existentially closed division rings. But the theory of division rings fails to have a model companion (see~\cite[Theorem 14.13]{HW}), so the same is true of the theory of Desarguesian projective planes. 

The same argument shows that, for any $d>2$, the class of $d$-dimensional projective spaces does not have a model companion; these spaces always satisfy Desargues's theorem, and hence are always coordinatized by division rings. On the other hand, the class of Pappian projective planes, which are exactly those coordinatized by fields, has a model companion, which is bi-interpretable with the theory of algebraically closed fields.

\section{Countable models}\label{sec:countable}

In this section, we consider the behavior of countable models of $T_{m,n}$. Recall that a countable theory $T$ is called \emph{small} if it has a countable saturated model (equivalently, $S_k(T)$ is countable for all $k>0$). We will show that $T_{m,n}$ is not small (unless $m=1$ or $n=1$, in which case $T_{m,n}$ is $\aleph_0$-categorical). On the other hand, for $m,n\geq 2$, the question of whether $T_{m,n}$ has a prime model seems to be a hard combinatorial problem. In the case of $T_{2,2}$, we show it is equivalent to a longstanding open problem in the theory of projective planes (see Theorem \ref{thm:primemodel}).

Let us first consider the special cases when $T_{m,n}$ is $\aleph_0$-categorical.

\begin{remark}\label{rem:m=1}
Models of $T_{m,1}$ are existentially closed incidence structures in which all lines contain exactly $m-1$ points. It follows that the algebraic closure of a set $A\subset\M$ is precisely $A$ together with all points incident to some line in $A$. Hence the algebraic closure of a finite set is finite, and $T$ is $\aleph_0$-categorical (using Corollary \ref{cor:acltype} to count types over $\emptyset$). Similarly, $T_{1,n}$ is $\aleph_0$-categorical. (From the point of view of classification theory, the theories $T_{m,1}$ and $T_{1,n}$ also behave very differently than the theories $T_{m,n}$ for $m,n\geq 2$, see Section \ref{subsec:stable}).
\end{remark}

In light of the previous remark, we focus on the case when $m,n\geq 2$. Toward showing $T_{m,n}$ is not small, the following lemma will be useful in allowing us to reduce to the case of $T_{2,2}$.

\begin{lemma}\label{lem:interpretable}
If $m_0\leq m$, $n_0\leq n$, and $M\models T_{m,n}$ then, using parameters, $M$ interprets a model of $T_{m_0,n_0}$.
\end{lemma}
\begin{proof}
In $M$ we may find $m-m_0$ points $\overline{c}=(c_1,\ldots,c_{m-m_0})$ and $n-n_0$ lines $\overline{d}=(d_1,\ldots,d_{n-n_0})$ such that the incidence structure on $\overline{cd}$ is $K_{m-m_0,n-n_0}$. Let $M_0$ be the induced substructure of $M$ such that $P(M_0)$ is the set of points in $M\backslash\overline{c}$ incident to all lines in $\overline{d}$, and $L(M_0)$ is the set of lines in $M\backslash\overline{d}$ incident to all points in $\overline{c}$. We show $M_0\models T_{m_0,n_0}$. 

Note that $M_0$ is an incidence structure by construction. Moreover, any copy of $K_{m_0,n_0}$ in $M_0$ would induce a copy of $K_{m,n}$ in $M$ by adding the points $\overline{c}$ and lines $\overline{d}$. So $M_0\models T^p_{m_0,n_0}$. To show $M_0\models T^c_{m_0,n_0}$, we need to show that any $m_0$ distinct points in $M_0$ lie on at least $n_0-1$ lines in $M_0$, and any $n_0$ distinct lines in $M_0$ intersect in at least $m_0-1$ points in $M_0$. We verify just the first assertion (the proof of the second is symmetric). Fix distinct points $x_1,\ldots,x_{m_0}\in M_0$. Then the $m$ points in $\overline{cx}$ simultaneously lie on $n-1$ distinct lines in $M$, $n-n_0$ of which are the lines in $\overline{d}$. The remaining $n_0-1$ lines are in $M_0$ by definition. Finally, to show $M_0\models T_{m_0,n_0}$, we fix a safe (with respect to $T_{m_0,n_0}$) diagram $\Delta(\overline{x};\overline{y})$ and $\overline{a}\in M_0$ such that $\widehat{\Delta}(\overline{a})$ holds. Let $\Delta_*(\overline{x},u_1,\ldots,u_{m-m_0},v_1,\ldots,v_{n-n_0};\overline{y})$ consist of the following formulas:
\begin{itemize}
\itemsep0em 
\parskip0em
\item all formulas in $\Delta(\overline{x};\overline{y})$,
\item $w\neq z$ for all distinct $w,z\in\overline{xuvy}$,
\item $P(u_i)$ and $\neg L(u_i)$ for all $i$,
\item $L(v_j)$ and $\neg P(v_j)$ for all $j$,
\item $I(u_i,v_j)$ for all $i,j$, 
\item $I(u_i,z)$ for all $i$ and all $z\in\overline{xy}$ such that $\Delta\models L(z)$, 
\item $I(z,v_j)$ for all $j$ and all $z\in\overline{xy}$ such that $\Delta\models P(z)$, and
\item $\neg I(w,z)$ for all $w,z\in\overline{xuvy}$ other than those for which $I(w,z)$ is specified above.
\end{itemize}
We claim that $\Delta_*(\overline{xuv};\overline{y})$ is safe with respect to $T_{m,n}$. It is clear that $\Delta_*$ defines an incidence structure on $\overline{xuvy}$. If $\Delta_*$ induces a copy of $K_{m,n}$, then it must contain at least $m_0$ points in $\overline{xy}$ and $n_0$ lines in $\overline{xy}$, which induces a copy of $K_{m_0,n_0}$ in $\Delta(\overline{x};\overline{y}$), contradicting that $\Delta$ is safe with respect to $T_{m_0,n_0}$. Finally, since $\Delta$ is safe with respect to $T_{m_0,n_0}$, any point in $\overline{y}$ is incident (via $\Delta$) to at most $n_0-1$ lines in $\overline{x}$, and thus incident (via $\Delta_*$) to at most $n-1$ lines in $\overline{xv}$. Similarly, any line in $\overline{y}$ is incident (via $\Delta_*$) to at most $m-1$ points in $\overline{xu}$.  Now, by construction we have $M\models\widehat{\Delta}_*(\overline{acd})$. So there is $\overline{b}\in M$ such that $M\models\widehat{\Delta}_*(\overline{acd};\overline{b})$. By construction $\overline{b}\in M_0$ and $M_0\models \Delta(\overline{a};\overline{b})$. 
\end{proof}

The remaining results in this section will involve detailed constructions of specific $K_{2,2}$-free incidence structures. So let us set some terminology and notation.  We refer to models of $T^p_{2,2}$ and $T^c_{2,2}$ as \emph{partial planes} and \emph{projective planes}, respectively (again, we do not include a ``non-degeneracy" condition). For clarity, we use $a,b,c$ for points and $\br,\bs,\bt$ for lines. Let $\Gamma$ be a partial plane. If distinct lines $\br$ and $\bs$ in $\Gamma$ contain a point $a\in\Gamma$, we say $a$ is the \emph{intersection point of $\{\br,\bs\}$ in $\Gamma$}. Similarly, if distinct points $a$ and $b$ in $\Gamma$ lie on a line $\br\in\Gamma$, we say $\br$ is the \emph{connecting line of $\{a,b\}$ in $\Gamma$}. If a pair of lines $\{\br,\bs\}$ (resp. a pair of points $\{a,b\}$) has no intersection point (resp. connecting line) in $\Gamma$, then we say that the pair is \emph{open in $\Gamma$}.

\begin{theorem}\label{thm:small}
If $m,n\geq 2$, then $T_{m,n}$ is not small.
\end{theorem}
\begin{proof}
By Lemma \ref{lem:interpretable}, we may assume $m=n=2$. We will show $S_4(T_{2,2})$ is uncountable.  We begin by constructing a partial plane $\Gamma_\emptyset$ as follows. Start with four distinct points $a_1,a_2,a_3,a_4$ and six distinct lines $\br_1,\br_2,\br_3,\br_4,\br_5,\br_6$, which are the connecting lines of the pairs $\{a_1,a_2\}$, $\{a_1,a_3\}$, $\{a_1,a_4\}$, $\{a_2,a_3\}$, $\{a_2,a_4\}$, $\{a_3,a_4\}$, respectively. Geometrically, we have a quadrangle consisting of four points together with six lines (the four sides and two diagonals).  There are three open pairs of lines: $\{\br_1,\br_6\}$, $\{\br_2,\br_5\}$, $\{\br_3,\br_4\}$, and we add distinct intersection points $b^0_1,b^0_2,b^0_3$ of the three pairs, respectively.  Since these three pairs of lines are mutually disjoint,  all pairs of points from the set $\{b^0_1,b^0_2,b^0_3\}$ are open. So we add distinct connecting lines $\bs^0_1,\bs^0_2,\bs^0_3$ of the pairs $\{b^0_1,b^0_2\}$, $\{b^0_1,b^0_3\}$, $\{b^0_2,b^0_3\}$, respectively. 

This finishes the construction of $\Gamma_\emptyset$. Explicitly, $P(\Gamma_\emptyset) = \{a_1,a_2,a_3,a_4, b^0_1,b^0_2,b^0_3\}$, $L(\Gamma_\emptyset) = \{\br_1,\br_2,\br_3,\br_4,\br_5,\br_6, \bs^0_1,\bs^0_2,\bs^0_3\}$, and $I(\Gamma_\emptyset)$ consists of the following incidences:
\[
\begin{tabular}{ccccccccc}
$(a_1,\br_1)$ & $(a_1,\br_2)$ &  $(a_1,\br_3)$ &  $(a_2,\br_4)$ &  $(a_2,\br_5)$ &  $(a_3,\br_6)$ & $(b^0_1,\bs^0_1)$ & $(b^0_1,\bs^0_2)$ & $(b^0_2,\bs^0_3)$ \\
$(a_2,\br_1)$ & $(a_3,\br_2)$ & $(a_4,\br_3)$ &   $(a_3,\br_4)$ &   $(a_4,\br_5)$ &  $(a_4,\br_6)$ &   $(b^0_2,\bs^0_1)$ &  $(b^0_3,\bs^0_2)$ & $(b^0_3,\bs^0_3)$ \\
$(b^0_1,\br_1)$ & $(b^0_2,\br_2)$ &  $(b^0_3,\br_3)$ &  $(b^0_3,\br_4)$ &  $(b^0_2,\br_5)$ &  $(b^0_1,\br_6)$. & & &
\end{tabular}
\]
By induction, we build a sequence $(\Gamma_\eta)_{\eta\in 2^{<\omega}}$ of partial planes such that, for all $\eta\in 2^{<\omega}$, 
\begin{enumerate}
\itemsep0em 
\parskip0em
\item if $\mu\preceq\eta$ then $\Gamma_\mu$ is an induced substructure of $\Gamma_\eta$;
\item $\{a_1,a_2,a_3,a_4\}$ generates $\Gamma_\eta$; 
\item if $k=|\eta|$ then $\Gamma_\eta$ contains three distinct lines $\bs^k_1,\bs^k_2,\bs^k_3$ such that each of the following six pairs of lines is open in $\Gamma_\eta$: $\{\br_1,\bs^k_3\}$, $\{\br_2,\bs^k_2\}$, $\{\br_3,\bs^k_1\}$, $\{\br_4,\bs^k_1\}$, $\{\br_5,\bs^k_2\}$, $\{\br_6,\bs^k_3\}$.
\end{enumerate}
By construction, $\Gamma_\emptyset$ satisfies these three conditions, and so we have the base case. For the induction step, fix $\eta\in 2^{<\omega}$ of length $k$, and suppose we have built $\Gamma_\eta$ satisfying conditions $(1)$, $(2)$, $(3)$ above. Fix $i\in\{0,1\}$. We first describe a two-step construction of $\Gamma_{\eta\widehat{~}(i)}$ by adding new points and lines to $\Gamma_\eta$. Along the way, we observe that each new point (resp. line) is added as an intersection point (resp. connecting line) of an open pair of lines (resp. points), and so the construction does not create a $K_{2,2}$. The first step of the construction will allow the induction to continue, and the second step distinguishes $\Gamma_{\eta\widehat{~}(0)}$ from $\Gamma_{\eta\widehat{~}(1)}$. After the two steps, we give the explicit definition of $\Gamma_{\eta\widehat{~}(i)}$. 

Step 1: Add distinct intersection points $b^{k+1}_1,b^{k+1}_2,b^{k+1}_3$ of the pairs $\{\br_1,\bs^k_3\}$, $\{\br_2,\bs^k_2\}$, $\{\br_3,\bs^k_1\}$, respectively, each of which is open by induction. Since these three pairs of lines are mutually disjoint, all pairs of points from the set $\{b^{k+1}_1,b^{k+1}_2,b^{k+1}_3\}$ are open. Now add distinct connecting lines $\bs^{k+1}_1,\bs^{k+1}_2,\bs^{k+1}_3$ of the pairs $\{b^{k+1}_1,b^{k+1}_2\}$, $\{b^{k+1}_1,b^{k+1}_3\}$, $\{b^{k+1}_2,b^{k+1}_3\}$, respectively. 

Step 2: Add distinct intersection points $c^{k+1}_1,c^{k+1}_2,c^{k+1}_3$ of the pairs $\{\br_4,\bs^k_1\}$, $\{\br_5,\bs^k_2\}$, $\{\br_6,\bs^k_3\}$, respectively, each of which has not yet been used and thus is still open by induction. Since these three pairs of lines are mutually disjoint, all pairs of points from the set $\{c^{k+1}_1,c^{k+1}_2,c^{k+1}_3\}$ are open. If $i=0$, add one final line $\bt^{k+1}$ containing all three points $c^{k+1}_1,c^{k+1}_2,c^{k+1}_3$. If $i=1$, add distinct intersection lines $\bt^{k+1}_1$ and $\bt^{k+1}_2$ of the pairs $\{c^{k+1}_1,c^{k+1}_2\}$ and $\{c^{k+1}_2,c^{k+1}_3\}$, respectively.

Explicitly, 
\begin{align*}
P(\Gamma_{\eta\widehat{~}(i)}) &= P(\Gamma_\eta)\cup \{b^{k+1}_1,b^{k+1}_2,b^{k+1}_3,c^{k+1}_1,c^{k+1}_2,c^{k+1}_3\},\\
L(\Gamma_{\eta\widehat{~}(i)}) &= L(\Gamma_\eta)\cup \{\bs^{k+1}_1,\bs^{k+1}_2,\bs^{k+1}_3\}\cup\begin{cases} \{\bt^{k+1}\} & \text{if $i=0$}\\ 
\{\bt_1^{k+1},\bt_2^{k+1}\} & \text{if $i=1$,}\end{cases}
\end{align*}
and $I(\Gamma_{\eta\widehat{~}(i)})$ is $I(\Gamma_\eta)$ together with the following incidences:
\[
\begin{tabular}{llllll}
$(b^{k+1}_1,\br_1)$ & $(b^{k+1}_2,\br_2)$ & $(b^{k+1}_3,\br_3)$ & $(c^{k+1}_1,\br_4)$ & $(c^{k+1}_2,\br_5)$ & $(c^{k+1}_3,\br_6)$ \\
$(b^{k+1}_3,\bs^k_1)$ & $(b^{k+1}_2,\bs^k_2)$ & $(b^{k+1}_1,\bs^k_3)$ & $(b^{k+1}_1,\bs^{k+1}_1)$ & $(b^{k+1}_1,\bs^{k+1}_2)$ &  $(b^{k+1}_2,\bs^{k+1}_3)$ \\
$(c^{k+1}_1,\bs^k_1)$ & $(c^{k+1}_2,\bs^k_2)$ & $(c^{k+1}_3,\bs^k_3)$ &  $(b^{k+1}_2,\bs^{k+1}_1)$ &  $(b^{k+1}_3,\bs^{k+1}_2)$ &  $(b^{k+1}_3,\bs^{k+1}_3)$
\end{tabular}
\]
and
\[
\left\{ 
\begin{tabular}{lllll}
$(c^{k+1}_1,\bt^{k+1})$ & $(c^{k+1}_2,\bt^{k+1})$ & $(c^{k+1}_3,\bt^{k+1})$ &  & if $i=0$,\\
$(c^{k+1}_1,\bt_1^{k+1})$ & $(c^{k+1}_2,\bt_1^{k+1})$ & $(c^{k+1}_2,\bt_2^{k+1})$ & $(c^{k+1}_3,\bt_2^{k+1})$  & if $i=1$.
\end{tabular}
\right.
\]

We now verify that $\Gamma_{\eta\widehat{~}(i)}$ satisfies conditions $(1)$, $(2)$, $(3)$. Condition $(1)$ is clear. For condition $(2)$, since each new point (resp. line) was introduced as an intersection point (resp. connecting line) of an existing pair of lines (resp. points), it follows by induction that $\{a_1,a_2,a_3,a_4\}$ generates $\Gamma_{\eta\widehat{~}(i)}$. Condition $(3)$ is easily checked by inspecting the list of new incidences above. 

This finishes the construction of the sequence $(\Gamma_\eta)_{\eta\in 2^{<\omega}}$. For any $\sigma\in 2^\omega$, $\Gamma_\sigma:=\bigcup_{n<\omega}\Gamma_{\sigma\restriction n}$ is a well-defined partial plane, which is generated by $\{a_1,a_2,a_3,a_4\}$. Therefore we have an $\cL$-embedding $f_\sigma\colon \Gamma_\sigma\to\M$. Given $\sigma\in 2^\omega$, define $p_\sigma=\tp(f_\sigma(a_1,a_2,a_3,a_4))\in S_4(T_{2,2})$. To finish the proof, it suffices to show that $\sigma\mapsto p_\sigma$ is injective.  To do this, we fix $k\geq 0$, $\eta\in 2^k$, and $\sigma,\tau\in 2^\omega$ such that $\eta\widehat{~}(0)\prec\sigma$ and $\eta\widehat{~}(1)\prec\tau$, and we show $p_\sigma\neq p_\tau$. Recall that, in $\M$, we have the definable function $H$ which sends any pair of lines (resp. points) to its unique intersection point (resp. connecting line). By construction, there are $H$-terms $u_i(x_1,x_2,x_3,x_4)$, for $i\in\{1,2,3\}$, such that $c^{k+1}_i=u_i(a_1,a_2,a_3,a_4)$ in both $\Gamma_\sigma$ and $\Gamma_\tau$. Then $p_\sigma\models H(u_1(\overline{x}),u_2(\overline{x}))=H(u_2(\overline{x}),u_3(\overline{x}))$, while $p_\tau\models H(u_1(\overline{x}),u_2(\overline{x}))\neq H(u_2(\overline{x}),u_3(\overline{x}))$.
\end{proof}

In direct analogy to the interpretation of division rings in Desarguesian projective planes described at the end of Section \ref{sec:props}, Hall~\cite{Hall} showed that the theory of projective planes with constants naming a quadrangle is bi-interpretable with the theory of planar ternary rings, a class of algebraic structures in a signature with a ternary function and two constants for $0$ and $1$. Just as before, the interpretations involve only existential formulas and put existentially closed projective planes in correspondence with existentially closed planar ternary rings.

Theorem~\ref{thm:small} shows that there are continuum-many theories of existentially closed projective planes with constants naming a quadrangle, which correspond to continuum-many theories of existentially closed planar ternary rings. In fact, all of these theories interpret $T_{2,2}$ and are interpretable (uniformly, depending on the choice of four parameters) in any sufficiently saturated model of $T_{2,2}$.

In contrast, it is a classical fact of projective geometry that in a Desarguesian plane, any two quadrangles (i.e., $4$-tuples of distinct points, no three of which are collinear) are conjugate by an automorphism. As a consequence, there are only finitely many quantifier-free $4$-types over the empty set in any Desarguesian plane, in the language with the function $H$ named.

In comparing this fact with Theorem~\ref{thm:small}, it is worth noting what Desargues's Theorem says about our construction. For each $k$, the triangles $(a_2,a_3,a_4;\br_4,\br_5,\br_6)$ and $(b^{k}_1,b^{k}_2,b^{k}_3;\bs^{k}_1,\bs^{k}_2,\bs^{k}_3)$ are in perspective, witnessed by the lines $\br_1$, $\br_2$, and $\br_3$ (which connect the pairs of points $(a_2,b^k_1)$, $(a_3,b^k_2)$, and $(a_4,b^k_3)$, respectively) meeting at the point $a_1$. As a consequence, Desargues's Theorem asserts that the points $c^{k+1}_1$, $c^{k+1}_2$, and $c^{k+1}_3$ (which are the intersections of the pairs of lines $(\br_4,\bs_1^k)$, $(\br_5,\bs_2^k)$, and $(\br_6,\bs_3^k)$, respectively) must be collinear, ruling out the independent choices that allowed us to find continuum-many $4$-types.

\begin{remark}\label{rem:lotsotypes}
Recall from Lemma \ref{lem:interpretable} that a model of $T_{2,2}$ is interpretable in $\M\models T_{m,n}$ using parameters isomorphic to $K_{m-2,n-2}$. Using this and the proof of Theorem~\ref{thm:small}, we make the following observations. Assume $m,n\geq 2$.
\begin{enumerate}
\item Algebraic closure in $T_{m,n}$ is not locally finite. Indeed, in the proof of Theorem~\ref{thm:small}, we have $\Gamma_\sigma\seq\acl(f_\sigma(a_1,a_2,a_3,a_4))$ for any $\sigma\in 2^{\omega}$.
\item $S_{m+n}(T_{m,n})$ is uncountable. Moreover, $T^c_{m,n}$ has continuum-many pairwise non-isomorphic models, each of which is generated by $m+n$ elements. These assertions follow by combining the continuum-many types in $S_4(T_{2,2})$ with the parameters $K_{m-2,n-2}$.
\item It is not hard to show that if $A\seq M\models T_{2,2}$  and $|A|\leq 3$, then $\acl(A)$ is finite. Hence, $k=4$ is minimal such that $S_k(T_{2,2})$ is infinite. This motivates the analogous questions for general $m,n$. If $A\seq M\models T_{m,n}$ has cardinality at most $m+n-1$, is it the case that $\acl(A)$ is finite? Is $m+n$ the minimal $k$ such that $S_k(T_{m,n})$ is infinite?
\end{enumerate}
\end{remark}

Having shown that $T_{m,n}$  has no countable saturated model when $m,n\geq 2$ (and in particular is not $\aleph_0$-categorical), the next natural question is whether $T_{m,n}$ has a prime model. In light of the fact that models of $T^c_{m,n}$ satisfy amalgamation (Proposition \ref{prop:amalgamation}), one might attempt to build a prime model via a Fra\"{i}ss\'{e}-style construction using finite models of $T^c_{m,n}$. Unfortunately, Proposition \ref{prop:amalgamation} does not guarantee that \emph{finite} models of $T^c_{m,n}$ satisfy amalgamation, since the proof uses free completion to extend a model of $T^p_{m,n}$ to a model of $T^c_{m,n}$. This motivates the following question.

\begin{question}\label{ques:Erdos}
Does every finite model of $T^p_{m,n}$ embed in a finite model of $T^c_{m,n}$?
\end{question}

Our use of the word ``embed" is in the model theoretic sense of induced substructures. Note that this question is a special case of the finite model property for $T_{m,n}$. Specifically, given a finite model $A\models T^p_m$, there is a sentence $\varphi_A\in T_{m,n}$ such that $M\models\varphi_A$ if and only if $M\models T^c_{m,n}$ and $A$ embeds in $M$ (namely, $\varphi_A$ is the conjunction of $\exists \overline{x}\Diag_A(\overline{x})$ with all axioms of $T^c_{m,n}$). We should also note that Question \ref{ques:Erdos} has a positive answer when $m=1$ or $n=1$ since, in these cases, the free completion of a finite model of $T^p_{m,1}$ or $T^p_{1,n}$ is finite. It is straightforward to show that $T_{m,1}$ and $T_{1,n}$ have the finite model property.

For $m,n\geq 2$, Question \ref{ques:Erdos} turns out to be a fairly difficult problem, even in the case $m=n=2$, where very little is known (see, e.g., \cite{MoWi}). Restated in the language of incidence geometry, the question for $m=n=2$ asks if every finite partial plane embeds in a finite projective plane. This problem was posed for matroids by Welsh in 1976 \cite[Chapter 12]{Welsh}, and for incidence geometries by  Erd\H{o}s in 1979 \cite[Problem 6.II]{erdos} (strictly speaking, Erd\H{o}s phrases the question in terms of non-induced embeddings, but it is an easy exercise to see that the two problems are equivalent \cite[Lemma 1]{MoWi}). The goal of the rest of this section is to show that the theory $T_{2,2}$ has a prime model if and only if this question about projective planes has a positive answer (and we expect this is also true for general $m,n\geq 2$; see Remark \ref{rem:primemodel}). The proof uses ``almost quantifier elimination" for $T_{m,n}$, along with the following combinatorial lemma.

\begin{lemma}\label{lem:notfree}
Suppose $A$ is a finite partial plane whose free completion $F(A)$ is infinite. Then there is a projective plane $B$ such that $A$ embeds in $B$ and generates $B$, but $B$ and $F(A)$ are not isomorphic over $A$.
\end{lemma}
\begin{proof}
We may assume that $F(A)$ is non-degenerate (i.e., it contains four points, no three of which are collinear). Indeed, the degenerate projective planes are completely classified (see~\cite{Hall}), and inspection shows that any finite subset  of a degenerate projective plane has finite $I$-closure. In any non-degenerate plane, there is a cardinal $\kappa$ such that every line contains $\kappa$ points and every point is on $\kappa$ lines. If $\kappa$ is finite, then the plane contains $(\kappa^2-\kappa+1)$ points and $(\kappa^2-\kappa+1)$ lines (and $\kappa-1$ is called the \emph{order} of the plane). Since $F(A)$ is non-degenerate and infinite, it must contain infinitely many points, each of which is incident with infinitely many lines. 

Now we inductively find lines $\br_1$, $\br_2$, $\dots$, $\br_7$ in $F(A)$ such that no three intersect in a single point and, if $k_i$ is minimal such that $\br_i\in F_{k_i}(A)$, then $k_1<\ldots<k_7$. Let $\br_1$ be any line in $F(A)$. Now suppose we have found $\br_i$ for $i\leq n<7$, and let $X = \{a_{i,j}\mid 1\leq i < j \leq n\}$, where $a_{i,j}\in F(A)$ is the intersection point of $\br_i$ and $\br_j$. Pick some point $d$ in $F(A)\backslash X$. Then each point in $X$ is incident with at most one line that is also incident to $d$. So, of the infinitely many lines incident to $d$, we may choose $\br_{n+1}$ to be not in $F_{k_n}(A)$ and not incident to any point in $X$.

Let $k=k_7$. In $F_k(A)$, $\br_7$ is incident with only two points, each of which lies on at most one $\br_i$ for $i\leq 6$. So there is some $i\in\{4,5,6\}$ such that the pair $\{\br_i,\br_7\}$ is open in $F_k(A)$. Let $b$ be the intersection point of $\{\br_i,\br_7\}$, added in $F_{k+1}(A)$. Note that the three pairs $\{a_{1,2},b\}$, $\{a_{1,3},b\}$, and $\{a_{2,3}, b\}$ are all open in $F_{k+1}(A)$, since none of $a_{1,2},a_{1,3},a_{2,3}$ is incident with $\br_i$ or $\br_7$. Moving to $F_{k+2}(A)$, we add distinct connecting lines $\bs_{1,2},\bs_{1,3},\bs_{2,3}$ of the pairs $\{a_{1,2},b\}$, $\{a_{1,3},b\}$, and $\{a_{2,3},b\}$, respectively. Now the pairs $\{\br_1,\bs_{2,3}\}$, $\{\br_2,\bs_{1,3}\}$, and $\{\br_3,\bs_{1,2}\}$ are open in $F_{k+2}(A)$. So, in $F_{k+3}(A)$, we add distinct intersection points $c_1,c_2,c_3$ of $\{\br_1,\bs_{2,3}\}$, $\{\br_2,\bs_{1,3}\}$, and $\{\br_3,\bs_{1,2}\}$, respectively. Since these three pairs of lines are mutually disjoint, all pairs of points from $\{c_1,c_2,c_3\}$ are open in $F_{k+3}(A)$. In the free completion $F(A)$, the three points $c_1$, $c_2$, and $c_3$ are not collinear, since distinct connecting lines for each pair are added in $F_{k+4}(A)$. But we can also extend $F_{k+3}(A)$ to a larger model $B_0$ of $T^p_{2,2}$ by adding a single new line $\bt$ incident to $c_1,c_2,c_3$. Note that $A$ still generates $B_0$, and thus generates the projective plane $F(B_0)$, which is not isomorphic to $F(A)$ over $A$. 
\end{proof}

\begin{remark}
In the proof of Lemma~\ref{lem:notfree}, the incidence structure $$\{a_{1,2},a_{1,3},a_{2,3},b,c_1,c_2,c_3;\br_1,\br_2,\br_3,\bs_{1,2},\bs_{1,3},\bs_{2,3},\bt\}\seq B$$ is a copy of the Fano plane. 
A finite partial plane is called a \emph{confined configuration} if each line is incident to at least three points and each point is incident to at least three lines. Hall observed in~\cite{Hall} that if $C\seq F(A)$ is a confined configuration, then $C\seq A$. The Fano plane is the smallest confined configuration, so it arises naturally as an obstruction to an isomorphism $B\cong F(A)$.
%Note that, by the proof of Theorem \ref{thm:small}, if $A$ is a quadrangle then there are uncountably many projective planes $B$, which are generated by $A$ and pairwise non-isomorphic over $A$.
\end{remark}

\begin{theorem}\label{thm:primemodel}$~$
\begin{enumerate}[$(a)$]
\item If every finite model of $T^p_{m,n}$ embeds in a finite model of $T^c_{m,n}$, then $T_{m,n}$ has a prime model.
\item $T_{2,2}$ has a prime model if and only if every finite partial plane embeds in a finite projective plane. 
\end{enumerate}
\end{theorem}
\begin{proof}
Recall that a countable complete theory has a prime model if and only if isolated types are dense, i.e., every formula is contained in a complete type which is isolated by a single formula.

Part $(a)$. Suppose that every finite model of $T^p_{m,n}$ embeds in a finite model of $T^c_{m,n}$. By Proposition~\ref{prop:partialqe}, every formula is equivalent to a finite disjunction of basic existential formulas, so it suffices to consider a basic existential formula $\exists\overline{y}\,\Delta_q(\overline{x};\overline{y})$, where $q$ is a complete quantifier-free type in variables $\overline{xy}$, such that $q$ implies that $\overline{x}$ generates $\overline{xy}$. Let $B$ be the finite model of $T^p_{m,n}$ determined by a realization $\overline{ab}$ of $q$, and let  $C$ be an extension of $B$ to a finite model of $T^c_{m,n}$. Replacing $C$ with the $I$-closure of $\overline{a}$ in $C$, we may assume that $C$ is generated by $\overline{a}$. Let $\overline{c}$ enumerate $C\backslash B$, and let $q'$ be the complete quantifier-free type of $\overline{abc}$. Then, by Corollary~\ref{cor:acltype}, the basic existential formula $\exists\overline{yz}\,\Delta_{q'}(\overline{x};\overline{y},\overline{z})$ isolates a complete type containing $\exists\overline{y}\,\Delta_q(\overline{x};\overline{y})$.

Part $(b)$. By part $(a)$, we only need the forward direction. Suppose the isolated types are dense relative to $T_{2,2}$, and suppose for contradiction that the finite partial plane $B$ does not embed in a finite projective plane. Let $\overline{x}$ enumerate $B$, and let $\varphi(\overline{x})$ be a formula isolating a complete type containing $\Diag_B(\overline{x})$. By Proposition~\ref{prop:partialqe}, we may assume that $\varphi(\overline{x})$ is basic existential, describing the quantifier-free type of a finite partial plane $C$. 

Now $B$ embeds in $C$, since $\varphi(\overline{x})$ implies $\Diag_B(\overline{x})$, and hence $C$ also fails to embed in a finite projective plane. In particular, the free completion $F(C)$ is infinite. By Lemma~\ref{lem:notfree}, $C$ (and hence also $B$) has multiple completions which are not isomorphic over $C$ (and hence not isomorphic over $B$). This contradicts the fact that $\varphi(\overline{x})$ isolates a complete type.   
\end{proof}

\begin{remark}\label{rem:primemodel}
We conjecture that the converse of Theorem \ref{thm:primemodel}$(a)$ also holds, in particular because the analog of Lemma \ref{lem:notfree} for general $m,n\geq 2$ ought to be true. The construction in the proof of this lemma generalizes in a straightforward manner for finite models $A\models T^p_{m,n}$ such that $F(A)$ contains infinitely many points and every $m-1$ points are simultaneously incident to infinitely many lines. Thus the hope is that models of $T^{c}_{m,n}$ failing this condition are somehow ``degenerate" in a classifiable way analogous to the $(2,2)$ case.  A more careful analysis of the combinatorics of models of $T^p_{m,n}$ would likely resolve this issue and yield answers to the questions in Remark \ref{rem:lotsotypes}$(3)$. 
\end{remark}

\section{Classification of $T_{m,n}$}\label{sec:NSOP1}

In this section, we analyze model theoretic properties of $T_{m,n}$, and we characterize various notions of independence including Kim independence, dividing independence, and forking independence. For background and details on these notions, see~\cite{casanovas} and~\cite{KRKim}.

In Section \ref{subsec:NSOP1}, we show that $T_{m,n}$ is always NSOP$_1$, and we characterize Kim independence using a combinatorial ternary relation defined over arbitrary base sets. In Section \ref{subsec:simple}, we characterize dividing independence in $T_{m,n}$ and use this to conclude that $T_{m,n}$ is not simple when $m,n\geq 2$. We also give an explicit instance of a formula witnessing TP$_2$ in this case. In Section \ref{subsec:forking}, we use the free completion of models of $T^p_{m,n}$ to define a stationary independence relation for $T_{m,n}$, which we then use to show that forking and dividing coincide for complete types (though there are formulas which fork but do not divide when $m,n\geq 2$). In Section \ref{subsec:ei}, we use the independence theorem for Kim independence, along with its interaction with algebraic independence, to prove that $T_{m,n}$ has weak elimination of imaginaries. We also analyze thorn-forking. Finally, in Section \ref{subsec:stable}, we observe that $T_{m,1}$ and $T_{1,n}$ are $\omega$-stable of rank $m$ and $n$, respectively.

In this section, we use $T$ to denote an arbitrary complete theory, and $\M$ for a monster model.

\subsection{Kim independence and NSOP$_1$}\label{subsec:NSOP1}

We will demonstrate NSOP$_1$ using a characterization due to Chernikov, Kaplan, and Ramsey \cite{ArtemNick,KRKim} involving the existence of a ternary relation satisfying certain axioms (see the proof of Theorem \ref{thm:NSOP1} below). We first recall the following classical ternary relations, as well as several axioms of an arbitrary ternary relation. 

\begin{definition}
Suppose $C\subset \M$ and $p$ is a global type (i.e., $p\in S(\mathbb{M})$). We say that $p$ is $C$\emph{-invariant} when, for every formula $\varphi(x,y)$, if $b \equiv_{C} b'$, then $\varphi(x,b) \in p$ if and only if $\varphi(x,b') \in p$ (equivalently, $p$ is invariant under the action of $\text{Aut}(\mathbb{M}/C)$ on $S(\mathbb{M})$).  If $p$ is a global $C$-invariant type, then a \emph{Morley sequence in $p$ over $C$} is a sequence $(b_{i})_{i \in \omega}$ from $\mathbb{M}$ so that $b_{i} \models p|_{C(b_{j})_{j < i}}$.
\end{definition}

\begin{definition}
Given $C\subset\M\models T$, and $a,b\in\M$, we define the following ternary relations:
\begin{enumerate}
\item \emph{Algebraic independence}: $a\trt{a}_C b$ if and only if $\acl(Ca)\cap\acl(Cb)=\acl(C)$.
\item \emph{Dividing independence}: $a\trt{d}_C b$ if and only if for any $C$-indiscernible $(b_i)_{i\in\omega}$ in $\tp(b/C)$, there is $a'$ such that $a'b_i\equiv_C ab$ for all $i\in\omega$.
\item \emph{Forking independence}: $a\trt{f}_C b$ if and only if for any tuple $\widehat{b}$, there is $a'\equiv_{Cb}a $ such that $a\trt{d}_C b\widehat{b}$. 
\item \emph{Kim dividing independence}: If $C$ is a model, then $a\trt{Kd}_C b$ if and only if for any Morley sequence $(b_i)_{i\in \omega}$ in a global $C$-invariant type extending $\tp(b/C)$, there is $a'$ such that $a'b_i\equiv_C ab$ for all $i\in \omega$.
\item \emph{Kim forking independence}: If $C$ is a model, then $a\trt{K}_C b$ if and only if for any tuple $\widehat{b}$, there is $a'\equiv_{Cb}a $ such that $a\trt{Kd}_C b\widehat{b}$. 
\end{enumerate}
Recall  that, for any $C\subset\M$ and $a,b\in\M$, we have $a\trt{f}_C b\Rightarrow a\trt{d}_C b\Rightarrow a\trt{a}_C b$, and if $C$ is a model, we have $a\trt{d}_C b\Rightarrow a\trt{Kd}_C b\Rightarrow a\trt{a}_C b$.
\end{definition}

If $T$ is NSOP$_1$, then $\trt{K}$ and $\trt{Kd}$ coincide over models (see \cite[Proposition 3.18]{KRKim}), in which case we suppress the notation $\trt{Kd}$ in favor of $\trt{K}$.  

\begin{definition}
Let $\gind$ be a ternary relation on (small subsets of) $\M\models T$. We say $\gind$ is \emph{invariant} if it is invariant under automorphisms of $\M$. Consider the following axioms.
\begin{itemize}
\item \emph{Monotonicity}: If $A\gind_C B$, $A'\seq A$, and $B'\seq B$, then $A'\gind_C B'$.
\item \emph{Base monotonicity}: If $A\gind_C B$ and $C\seq D\seq\acl(BC)$ then $A\gind_D B$.
\item \emph{Symmetry}: If $A\gind_C B$ then $B\gind_C A$.
\item \emph{Transitivity}: If $D\seq C\seq B$, $A\gind_D C$, and $A\gind_C B$ then $A\gind_D B$.
\item  \emph{Finite character}: If $A\ngind_C B$ then $A\ngind_C B_0$ for some finite $B_0\seq B$. 
\item \emph{Local character}: For all $A$ there is some cardinal $\kappa$ such that for all $B$, there is $C\seq B$ with $|C|\leq\kappa$ and $A\gind_C B$. 
\item \emph{Existence}: $A\gind_C C$ for all $A$ and $C$.
\item \emph{Full existence}: For all $A,B,C$ there is $A'\equiv_C A$ such that $A'\gind_C B$.
\item \emph{Extension}: If $A\gind_C B$ and $\widehat{B}\supseteq BC$, then there is $A'\equiv_{BC} A$ such that $A'\gind_C \widehat{B}$.
\item \emph{Stationarity}: For all $A,A',B,C$, if $A'\equiv_C A$, $A\gind_C B$, and $A'\gind_C B$, then $A'\equiv_{BC} A$.
\item \emph{The independence theorem}: If $M\models T$ and $A\gind_M B$, $A'\gind_M C$, $B\gind_M C$, and $A\equiv_M A'$, then there is $A''$ such that $A''\equiv_{MB} A$, $A''\equiv_{MC}A'$, and $A''\gind_M BC$.
\item \emph{Strong finite character and witnessing}: If $M\models T$ and $A\ngind_M B$, then there are $a\in A$, $b\in B$, and $\varphi(x,b)\in\tp(a/BM)$, such that if $a'\models\varphi(x,b)$ then $a'\ngind_M b$, and if $(b_i)_{i\in\omega}$ is a Morley sequence in a global invariant type extending $\tp(b/M)$, then $\{\varphi(x,b_i) \mid i\in\omega\}$ is inconsistent.
\end{itemize}
\end{definition}

We will use the following axiomatic characterization of Kim independence in NSOP$_1$ theories.

\begin{theorem}\label{thm:characterizingkim}\textnormal{\cite[Theorem 9.1]{KRKim}}
Suppose $\gind$ is an invariant ternary relation on $\M$ which satisfies strong finite character and witnessing, the independence theorem, existence over models, monotonicity over models, and symmetry over models, then $T$ is NSOP$_1$ and $\gind = \trt{K}$.
\end{theorem}

We now define the ternary relation which will be used to prove that $T_{m,n}$ is NSOP$_1$. 

\begin{definition}\label{def:KI}
Define $\iind$ on subsets of $\mathbb{M}\models T_{m,n}$ by $A\iind_C B$ if and only if
 \[
\textstyle A\trt{a}_C B\text{ and, if $a\in \acl(AC)$ and $b\in\acl(BC)$ are incident, then $a\in\acl(C)$ or $b\in\acl(C)$.}\]
\end{definition}

Note that the knowledge of $\tp(A/C)$ and $\tp(B/C)$, together with $A\iind_C B$, does not uniquely determine $\tp(AB/C)$; that is, $\iind$ is not stationary. Specifically, the behavior of the elements of $\acl(ABC)\backslash (\acl(AC)\cup \acl(BC))$ is not specified. Throughout this section, we will implicitly use the observation that, given $A,B,C\subset\M\models T_{m,n}$, $A\iind_C B$ if and only if $\acl(AC)\iind_{\acl(C)}\acl(BC)$.

\begin{definition}
Given a ternary relation $\gind$ on $\M\models T$, a set $C\subset\M$, and a tuple $b\in\M$, an \emph{$\gind$-independent sequence in $\tp(b/C)$} is a sequence $(b_i)_{i\in\omega}$ such that, for all $i\in\omega$, $b_i\equiv_C b$ and $b_i\gind_C b_{<i}$.
\end{definition}

The following easy fact establishes strong finite character and witnessing for $\trt{a}$ (since any Morley sequence in a global invariant type is $\trt{a}$-independent). See also~\cite[Lemma 2.7]{KrRa}.

\begin{fact}\label{fact:SFCWalg}
Given $C\subset\M\models T$ and $a,b\in\M$, if $a\ntrt{a}_C b$ then there is an $\cL_C$-formula $\varphi(x,y)$ such that:
\begin{enumerate}[$(i)$]
\item $\varphi(x,b)\in\tp(a/bC)$,
\item if $a'\models\varphi(x,b)$, then $a'\ntrt{a}_C b$,
\item if $(b_i)_{i\in\omega}$ is an $\trt{a}$-independent sequence in $\tp(b/C)$, then $\{\varphi(x,b_i):i\in\omega\}$ is inconsistent.
\end{enumerate}
\end{fact}

We need a version of this fact to establish strong finite character and witnessing for $\iind$ in $T_{m,n}$. Condition $(iii)$ will be used later in Section~\ref{subsec:ei}.

\begin{lemma}\label{lem:SFCW}
Given $C\subset\M\models T_{m,n}$ and $a,b\in\M$, if $a\niind_C b$ then there is an $\cL_C$-formula $\varphi(x,y)$ such that:
\begin{enumerate}[$(i)$]
\item $\varphi(x,b)\in\tp(a/bC)$,
\item if $a'\models\varphi(x,b)$, then $a'\niind_C b$,
\item if $(b_i)_{i\in\omega}$ is an $\trt{a}$-independent sequence in $\tp(b/C)$ and $a'$ realizes $\{\varphi(x,b_i):i\in\omega\}$, then $a'\ntrt{a}_C (b_i)_{i\in\omega}$.
\item if $(b_i)_{i\in\omega}$ is an $\iind$-independent sequence in $\tp(b/C)$, then $\{\varphi(x,b_i):i\in\omega\}$ is inconsistent.
\end{enumerate}
\end{lemma}
\begin{proof}
If $a\ntrt{a}_C b$, then we have $(i)$, $(ii)$, and $(iv)$ by Fact~\ref{fact:SFCWalg} (since $\ntrt{a}$ implies $\niind$, and $\iind$-independent sequences are $\trt{a}$-independent), and $(iii)$ holds trivially. We may therefore assume that there is an incidence between some $u\in\acl(Ca)\backslash\acl(C)$ and $v\in\acl(Cb)\backslash\acl(C)$. Without loss of generality, $u$ is a point and $v$ is a line (the other case is symmetric). Fix $\cL_C$-formulas $\psi(w,x)$ and $\theta(z,y)$, and integers $k,\ell$, such that $u$ is one of exactly $k$ solutions to $\psi(w,a)$ and $v$ is one of exactly $\ell$ solutions to $\theta(z,b)$. We may assume $\theta(z,b)$ has no solutions in $\acl(C)$, and that $\theta(z,y)$ includes $L(z)$ as a conjunct. Therefore, if $\theta(z,b)$ holds, then $z$ is a line and there are most $m-1$ points in $\acl(C)$ incident to $z$. So we may fix an algebraic $\cL_C$-formula $\chi(w)$ such that any point in $\acl(C)$, which lies on a line satisfying $\theta(z,b)$, is a solution to $\chi(w)$. Now let $\varphi(x,y)$ be the $\cL_C$-formula:
\[
\exists^{!k}w\,\psi(w,x)\land\exists^{!\ell}z\,\theta(z,y)\land\exists w\exists z(\psi(w,x)\land \theta(z,y)\land I(w,z)\land \neg \chi(w)).
\]
 Then $\varphi(x,b)\in\tp(a/Cb)$ by construction. Suppose $a'\models\varphi(x,b)$, witnessed by $w,z$. Then $w\in\acl(Ca')$, $z\in\acl(Cb)$, and $I(w,z)$. We also have $z\not\in\acl(C)$ since $\theta(z,b)$ holds, and $w\not\in\acl(C)$, since $I(w,z)\land \neg\chi(w)\land \theta(z,b)$ holds. 
 
For $(iii)$, suppose $(b_i)_{i\in\omega}$ is an $\trt{a}$-independent sequence in $\tp(b/C)$, and suppose $a'$ realizes $\{\varphi(x,b_i)\mid i\in\omega\}$. After translating by an automorphism over $C$, we may assume $b_0=b$.  For each $i\in\omega$, $\varphi(a',b_i)$ is witnessed by some $w_i$ realizing $\psi(w,a')$ and $z_i$ realizing $\theta(z,b_i)$. Moreover, $\psi(w,a')$ has at most $k$ solutions. By pigeonhole, we may pass to a subsequence and assume $w_i=w_j=:w$ for all $i,j\in\omega$. For any $i\in\omega$, $b_i\equiv_C b$ and so $\theta(z,b_i)$ has exactly $\ell$ solutions, none of which lie in $\acl(C)$. In particular $z_i\in\acl(Cb_i)\backslash\acl(C)$ for all $i\in\omega$. Since $(b_i)_{i\in\omega}$ is $\trt{a}$-independent, it follows that $z_i\neq z_j$ for all distinct $i,j\in\omega$. On the other hand, $w$ is incident to $z_i$ for all $i\in\omega$. In particular, $w$ is one of $m-1$ points which simultaneously lie on $z_0,\ldots,z_{n-1}$, and so $w\in\acl(Cb_0\ldots b_{n-1})$. But also $w\in \acl(Ca')$ and $w\notin \acl(C)$, since $\theta(z_0,b)$, $I(w,z_0)$, and $\lnot \chi(w)$ hold. So $w$ witnesses $a'\ntrt{a}_C b_{<n}$.  

For $(iv)$, we may continue with the further assumption that $(b_i)_{i\in\omega}$ is an $\iind$-independent sequence. But we have  $z_n\in\acl(Cb_n)\backslash\acl(C)$, $w\in \acl(Cb_{<n})$, and $I(w,z_n)$, which contradicts $b_n\iind_C b_{<n}$.
\end{proof}

\begin{remark}\label{rem:SFCW}
It follows from Lemma \ref{lem:SFCW} that $\iind$ satisfies strong finite character and witnessing in $T_{m,n}$. This is because, if $M\models T_{m,n}$, then any Morley sequence in a global invariant type extending $\tp(b/M)$ is $\iind$-independent (e.g.\ by Corollary \ref{cor:div} below, and since if $\tp(a/Mb)$ extends to a global invariant type over $M$, then $a\trt{d}_M b$). Alternatively, as discussed in \cite{KRKim}, when using the witnessing axiom to demonstrate NSOP$_1$, it is enough to assume that the global invariant type is finitely satisfiable in $M$. Moreover, if $\tp(a/Mb)$ is finitely satisfiable in $M$, then it is easy to directly see $a\iind_M b$, and so this version of witnessing follows immediately from the previous lemma (without needing Corollary \ref{cor:div}).
\end{remark}

\begin{lemma}\label{lem:IT}
Suppose $D = \acl(D) \subset \M\models T_{m,n}$ and $a,a',b,c\in\M$ are such that $a\iind_D b$, $a'\iind_D c$, $b\trt{a}_D c$, and $a\equiv_D a'$. Then there is $a''$ such that $a''\equiv_{Db} a$, $a''\equiv_{Dc} a'$, and $a''\iind_D bc$. 
\end{lemma}
\begin{proof}
Let $X_{ab} = \acl(Dab)$, $X_{ac} = \acl(Da'c)$, and $X_{bc} = \acl(Dbc)$. Since $a\equiv_D a'$, $\acl(Da)\seq X_{ab}$ and $\acl(Da')\seq X_{ac}$ are isomorphic via a map sending $a$ to $a'$ and fixing $D$ pointwise. Let $X_{abc}$ be the incidence structure formed by taking the disjoint union of $X_{ab}$, $X_{ac}$, and $X_{bc}$ (adding no new incidences), and then identifying the copies of $X_b = \acl(Db)$ in $X_{ab}$ and $X_{bc}$, identifying the copies of $X_c = \acl(Dc)$ in $X_{ac}$ and $X_{bc}$, and identifying the isomorphic copies of $\acl(Da)$ in $X_{ab}$ and $\acl(Da')$ in $X_{ac}$ (call the identified substructure $X_a$). Note that our independence assumptions imply that $X_a$, $X_b$, and $X_c$ are pairwise disjoint over $D = \acl(D)$, and the identifications described above agree on their common intersection $D$. By construction, we have $X_r \trt{a}_D X_{uv}$ whenever $\{r,u,v\} = \{a,b,c\}$.

Then $X_{abc}\models T^p_{m,n}$. Indeed, suppose for contradiction that there is a copy $\overline{wz}$ of $K_{m,n}$ in $X_{abc}$, where $\overline{w}=(w_1,\ldots,w_m)$ is a tuple of points, $\overline{z}=(z_1,\ldots,z_n)$ is a tuple of lines, and $I(w_i,z_j)$ for all $i,j$. First, we claim that $\overline{wz}$ is not contained in $X_a\cup X_b\cup X_c$. Indeed, suppose otherwise. Then, since $\overline{wz}$ is not contained in $X_a$, there is an element of $\overline{wz}$ in $X_b\backslash D$ or $X_c\backslash D$. Without loss of generality, suppose $X_b\backslash D$ contains a point $w_i$ in $\overline{w}$ (the other cases are symmetric). Since $a\iind_D b$, we must have $\overline{z}$ contained in $X_b\cup X_c$. If $\overline{z}$ is entirely in $X_c$ then, viewing $w_i$ and $\overline{z}$ in the structure $X_{bc}$, we have $w_i$ in $(\acl(Db)\cap\acl(Dc))\backslash D$, which contradicts $b\trt{a}_D c$. So $X_b\backslash D$ must also contain a line $z_j$ in $\overline{z}$. Again, $a\iind_D b$ then implies $X_b\cup X_c$ contains $\overline{w}$, and thus contains $\overline{wz}$, which is a contradiction. This establishes the claim. 

Hence there is a pair $(u,v)$ from $\{a,b,c\}$ such that $X_{uv}\backslash(X_u\cup X_v)$ contains an element of $\overline{wz}$. Without loss of generality, we assume this is a point $w_i$ in $\overline{w}$. Since no new incidences were introduced in the construction of $X_{abc}$, we must have $\overline{z}$ contained in $X_{uv}$. Note that there is no line $z_j$ in $X_{uv}\backslash(X_u\cup X_v)$. Otherwise, by the same argument, we would also have $\overline{w}$, and thus all of $\overline{wz}$, contained in $X_{uv}$. So $\overline{z}$ is contained in $X_u\cup X_v$. But we also cannot have all of $\overline{z}$ contained either $X_u$ or $X_v$, since then, viewing $w_i$ and $\overline{z}$ in the structure $X_{uv}$, we would have $w_i\in \acl(\overline{z})$, but $w_i\not\in X_u\cup X_v$. It follows that there are distinct lines $z_j\in X_u\backslash D$ and $z_k\in X_v\backslash D$. 

Let $r$ be the unique element of $\{a,b,c\}\backslash \{u,v\}$. Since $\overline{wz}$ is not contained in $X_{uv}$, there is a point $w_\ell\not\in X_{uv}$. Again, since no new incidences were introduced in the construction of $X_{abc}$, and $w_\ell$ lies on both $z_j$ and $z_k$, we must have $w_\ell\in X_r\backslash D$. But then there is an incidence between $X_r\backslash D$ and $X_u\backslash D$, and between $X_r\backslash D$ and $X_v\backslash D$. Since one of the pairs $\{r,u\}$ and $\{r,v\}$ must be $\{a,b\}$ or $\{a,c\}$, we have contradicted one of $a\iind_D b$ or $a'\iind_D c$.

Having shown that $X_{abc}\models T^p_{m,n}$, we can embed $X_{abc}$ into $\M$ over $X_{bc}$ by Proposition \ref{prop:extend}. Letting $a''$ be the image of the tuple $a$ in $X_a$ under this embedding, we have $a''\equiv_{Db} a$ and $a''\equiv_{Dc} a'$ by Corollary \ref{cor:extend}. And $a''\iind_D bc$, since there are no incidences between $X_a\backslash D$ and $X_{b}\backslash D$, no incidences between $X_a\backslash D$ and $X_{c}\backslash D$, and no additional incidences added between $X_a\backslash D$ and $X_{bc}$ when we built $X_{abc}$.
\end{proof}

\begin{theorem}\label{thm:NSOP1}
$T_{m,n}$ is NSOP$_1$ and $\iind = \trt{K}=\trt{Kd}$ over models.
\end{theorem}
\begin{proof}
To prove that $T_{m,n}$ is NSOP$_1$ and $\iind = \trt{K}$, it suffices to show that $\iind$ satisfies the axioms given in Theorem~\ref{thm:characterizingkim}. Then also $\trt{K} = \trt{Kd}$ by~\cite[Proposition 3.18]{KRKim}.

By definition, $\iind$ clearly satisfies existence, monotonicity, and symmetry. We also have strong finite character and witnessing by Lemma \ref{lem:SFCW} and Remark \ref{rem:SFCW}. Finally, we have the independence theorem by Lemma \ref{lem:IT} and since $\iind$ implies $\trt{a}$.
\end{proof}

\begin{remark}\label{rem:FE}
While the existence axiom for $\iind$ is sufficient in the proof of NSOP$_1$, below we will need to use the fact that $\iind$ satisfies full existence. Of course, $\trt{K}$ satisfies full existence over models in any NSOP$_1$ theory~\cite[Proposition 3.19]{KRKim}, but we'll give a proof for $\iind$ that works over arbitrary subsets of $\M\models T_{m,n}$. Fix $C\subset\M$ and $a,b\in\M$ and let $A=\acl(Ca)$, $B=\acl(Cb)$, and $C_*=\acl(C)$. We have the inclusion maps $f\colon C_*\to A$ and $g\colon C_*\to B$. By the proof of Proposition \ref{prop:amalgamation}, there is $D\models T^p_{m,n}$ and embeddings $f_*\colon A\to D$ and $g_*\colon B\to D$ such that $f_*|C_*=g_*|C_*$, $f_*(A)\cap f_*(B)=f_*(C_*)$ and there are no incidences between $f_*(A)\backslash f_*(C_*)$ and $g_*(B)\backslash f_*(C_*)$. By Proposition \ref{prop:extend}, there is $\widehat{g}\colon D\to\M$ such that $\widehat{g}g_*$ is inclusion on $B$. Let $h=\widehat{g}f_*$, and so $h$ is an $\cL$-embedding from $A$ to $\M$ fixing $C$ pointwise. If $a'=h(a)$, then $a'\equiv_C a$ by Corollary \ref{cor:extend}.  By construction $\widehat{g}f_*(A)\iind_{C_*}B$. Since $\widehat{g}f_*(A)=\acl(a'C)$, we have $a'\iind_C b$.
\end{remark}

\subsection{Dividing independence and simplicity}\label{subsec:simple}

Our next goal is to show that if $m,n\geq 2$ then $T_{m,n}$ is not simple, and so is \emph{properly} NSOP$_1$. In fact, an explicit witness of TP$_2$ for $T_{m,n}$ is not difficult to construct (see Proposition \ref{prop:TP2} below). However, it will be useful and interesting to first develop a few more properties of $\iind$, which will allow us to characterize dividing independence.

\begin{proposition}\label{prop:Kequiv}
Given $C = \acl(C) \subset\M\models T_{m,n}$ and $a,b\in\M$, the following are equivalent:
\begin{enumerate}[$(i)$]
\item $a\iind_C b$;
\item for any $\trt{a}$-independent $(b_i)_{i\in\omega}$ in $\tp(b/C)$ there is $a'$ such that $a'b_i\equiv_C ab$ for all $i\in\omega$;
\item for any $\iind$-independent $(b_i)_{i\in\omega}$ in $\tp(b/C)$ there is $a'$ such that $a'b_i\equiv_C ab$ for all $i\in\omega$;
\item for some $\iind$-independent $(b_i)_{i\in\omega}$ in $\tp(b/C)$ there is $a'$ such that $a'b_i\equiv_C ab$ for all $i\in\omega$.
\end{enumerate}
Moreover, if $(ii)'$, $(iii)'$, and $(iv)'$ are obtained from $(ii)$, $(iii)$, and $(iv)$ by adding the assumption that $(b_i)_{i\in\omega}$ is $C$-indiscernible, then $(ii)'$, $(iii)'$, and $(iv)'$ are also equivalent to $(i)$ through $(iv)$. 
\end{proposition}
\begin{proof}
$(ii)\Rightarrow(iii)$ is trivial since $\iind$ implies $\trt{a}$. $(iii)\Rightarrow (iv)$ is trivial since $\iind$-independent sequences exist by full existence for $\iind$ (see Remark \ref{rem:FE}). $(iv)\Rightarrow (i)$ is by Lemma \ref{lem:SFCW}.

$(i)\Rightarrow(ii)$. Fix $\trt{a}$-independent $(b_i)_{i\in\omega}$ in $\tp(b/C)$. Without loss of generality,  $b=b_0$. For $i\in\omega$, let $a_i$ be such that $a_ib_i\equiv_C ab$. By induction, we construct $(a'_n)_{n\in\omega}$ such that, for all $n\in\omega$, $a'_n\iind_C b_{\leq n}$ and $a'_nb_i\equiv_C ab$ for all $i\leq n$. The desired result then follows by compactness. For the base case, let $a'_0=a_0$. Assume we have constructed $a'_n$ as desired. Then we have $a'_n\iind_C b_{\leq n}$, $a_{n+1}\iind_C b_{n+1}$, $b_{n+1}\trt{a}_C b_{\leq n}$, and $a'_n\equiv_C a_{n+1}$. By Lemma \ref{lem:IT}, there is $a'_{n+1}$ such that $a'_{n+1}\equiv_{Cb_{\leq n}}a'_n$, $a'_{n+1}\equiv_{Cb_{n+1}}a_{n+1}$, and $a'_{n+1}\iind_C b_{\leq n+1}$. Then $a'_{n+1}b_{n+1}\equiv_C a_{n+1}b_{n+1}\equiv_C ab$ and, for $i\leq n$, $a'_{n+1}b_i\equiv_C a'_nb_i\equiv_C ab$ by induction. So $a'_{n+1}$ is as desired.

For the moreover statement, note that $(ii)\Rightarrow (ii)'$ and $(iv)'\Rightarrow(iv)$ are trivial, and $(ii)'\Rightarrow (iii)'$ is again immediate since $\iind$ implies $\trt{a}$. To finish the equivalences, it suffices to show $(iii)'\Rightarrow (iv)'$, which is immediate since a $C$-indiscernible $\iind$-independent sequence in $\tp(b/C)$ can be extracted from an $\iind$-independent sequence in $\tp(b/C)$ in the usual way. 
\end{proof}

From Proposition \ref{prop:Kequiv}, we obtain a combinatorial characterization of $\trt{d}$ as the ternary relation obtained by ``forcing" base monotonicity on $\iind$.

\begin{corollary}\label{cor:div}
In $T_{m,n}$, $\trt{d}$ implies $\iind$. Moreover, for any $C\subset\M\models T_{m,n}$ and $a,b\in\M$,
\[
\textstyle a\trt{d}_C b ~\Leftrightarrow~a\iind_D b\text{ for all $D$ such that $C\seq D\seq\acl(Cb)$.}
\]
\end{corollary}
\begin{proof}
First, note that $\trt{d}$ implies $\iind$ by definition of $\trt{d}$ and Proposition \ref{prop:Kequiv}. Since $\trt{d}$ satisfies base monotonicity in any theory, we have the forward direction of the desired equivalence. For the converse direction, suppose $a\ntrt{d}_ C b$. We may assume $b$ and $C$ are algebraically closed and $C$ is contained in $b$. Fix a $C$-indiscernible sequence $(b_i)_{i\in \omega}$ witnessing $a\ntrt{d}_C b$, and let $d$ be the subtuple of $b$ which is constant along the indiscernible sequence. Then $d$ is algebraically closed, $C$ is contained in $d$, and $(b_i)_{i\in \omega}$ is an $\trt{a}$-independent sequence in $\tp(b/d)$. By choice of the sequence, there is no $a'$ such that $a'b_i\equiv_{d} ab$ for all $i\in \omega$. So $a\niind_{d} b$ by Proposition \ref{prop:Kequiv}.
\end{proof}

\begin{remark}
Proposition \ref{prop:Kequiv} and Corollary \ref{cor:div} still hold if we replace $\iind$ by an invariant ternary relation on $\M$, which implies $\trt{a}$ and satisfies full existence and the variations of strong finite character and witnessing and the independence theorem given by Lemmas \ref{lem:SFCW} and \ref{lem:IT}.
\end{remark}

Given Corollary \ref{cor:div}, the natural question is whether, in $T_{m,n}$, $\iind$ satisfies base monotonicity  in the first place, which would then mean that $\trt{d}$ and $\iind$ coincide. But this is not the case.

\begin{proposition}\label{prop:BM}
If $m,n\geq 2$, then $\iind$ fails base monotonicity in $T_{m,n}$.
\end{proposition}
\begin{proof}
Define an incidence structure $\Gamma$ such that 
\begin{align*}
P(\Gamma) &= \{a_1,b,w_1,\ldots,w_{m-1}\},\\
L(\Gamma) &= \{a_2,c_1,\ldots,c_{n-1},z\},\text{ and }\\
I(\Gamma) &= \{(a_1,z),(b,z)\}\cup\{(w_i,a_2),(w_i,c_j),(w_i,z):i<m,~j<n\}.
\end{align*}
We first note that $\Gamma$ is $K_{m,n}$-free, since $z$ is the only line in $\Gamma$ which contains more than $m-1$ points. So we may assume $\Gamma\subset\M$. Let $\overline{c}=(c_1,\ldots,c_{n-1})$ and $\overline{w}=(w_1,\ldots,w_{m-1})$. By Proposition \ref{prop:Iacl}, $a_1a_2$ and $b\overline{c}$ are each algebraically closed and so, by construction, $a_1a_2\iind_\emptyset b\overline{c}$. However, $a_1a_2\niind_{\overline{c}} b\overline{c}$. Indeed, we have $I(w_i,v)$ for all $v\in a_2\overline{c}$ and $i<m$ and so $\overline{w}\seq\acl(a_1a_2\overline{c})$. We also have $I(v,z)$ for all $v\in a_1\overline{w}$ and so, altogether, $z\in\acl(a_1a_2\overline{c})$. Therefore $I(b,z)$ holds, $z\in\acl(a_1a_2\overline{c})\backslash\acl(\overline{c})$, and $b\in\acl(b\overline{c})\backslash\acl(\overline{c})$, which gives $a_1a_2\niind_{\overline{c}}b\overline{c}$, as desired. 
\end{proof}

We can use this observation, together with general results on NSOP$_1$ theories, to make the following conclusion.

\begin{corollary}
If $m,n\geq 2$ then $T_{m,n}$ is not simple.
\end{corollary}
\begin{proof}
By \cite[Proposition 8.4]{KRKim}, an NSOP$_1$ theory is simple if and only if $\trt{K}$ and $\trt{d}$ coincide over models. So fix $M\models T_{m,n}$ and let $\Gamma$ be as in the proof of Proposition \ref{prop:BM}, and also such that $\Gamma\cap M=\emptyset$ and $a_1a_2\iind_M b\overline{c}$. Then $a_1a_2\trt{K}_M b\overline{c}$ by Theorem \ref{thm:NSOP1}. As in the proof of Proposition \ref{prop:BM}, $a_1a_2\niind_{M\overline{c}} b\overline{c}$ and so $a_1a_2\ntrt{d}_{M\overline{c}}b\overline{c}$ by Corollary \ref{cor:div}. Since $\trt{d}$ satisfies base monotonicity, it follows that $a_1a_2\ntrt{d}_M b\overline{c}$. Altogether, $\trt{d}_M\neq \trt{K}_M$. 
\end{proof}

On the other hand, we can also use the structure $\Gamma$ in the proof of Proposition \ref{prop:BM} to give an explicit example of TP$_2$ in $T_{m,n}$. In particular, let $\varphi(x_1,x_2;y_1,\ldots,y_n)$ be the $\cL$-formula such that $\varphi(a_1,a_2;b,\overline{c})$ holds if and only if there exist $w_1,\ldots,w_{m-1}$ and $z$ so that the incidence structure on $a_1a_2b\overline{c}\overline{w}z$ is precisely $\Gamma$. The next result demonstrates TP$_2$ for $\varphi$. This is somewhat superfluous, since any properly NSOP$_1$ theory must have TP$_2$. However, the properties of $\varphi$ shown in the proof will be used later in Proposition \ref{prop:forkingdividing}. 

\begin{proposition}\label{prop:TP2}
If $m,n\geq 2$ then $\varphi(x_1,x_2;y_1,\ldots,y_n)$ has TP$_2$ in $T_{m,n}$.
\end{proposition}
\begin{proof}
Let $a_1,a_2,b,\overline{c}\in\M$ be such that $\varphi(a_1,a_2;b,\overline{c})$ holds (these elements exist since $\Gamma$ is $K_{m,n}$-free).  As above, $a_1a_2\iind_\emptyset b\overline{c}$. Let $b_0=b$ and, using full existence for $\iind$, construct a sequence $(b_i)_{i\in\omega}$ such that, for all $i<\omega$, $b_i\equiv_{\overline{c}}b_0$ and $b_i\iind_{\overline{c}}b_{<i}$. We claim that $\{\varphi(x_1,x_2;b_i,\overline{c}) \mid i\in\omega\}$ is $r$-inconsistent, where $r=(m+1)(n-1)$. Indeed, suppose we have $a'_1,a'_2\in\M$ such that $\varphi(a'_1,a'_2;b_{i_j},\overline{c})$ holds for some $i_1<\ldots<i_r$, witnessed by $w^j_1,\ldots,w^j_{m-1}$ and $z_j$ for $1\leq j\leq r$. For a fixed $1\leq j\leq r$, the elements $w^j_1,\ldots,w^j_{m-1}$ are $m-1$ distinct points which simultaneously lie on the distinct lines $a'_2,c_1,\ldots,c_{n-1}$. It follows that, for $1\leq j<k\leq r$, we have $\{w^j_1,\ldots,w^j_{m-1}\}=\{w^k_1,\ldots,w^k_{m-1}\}$. After relabeling, we may assume $w^j_i=w^k_i=:w_i$ for all $1\leq j<k\leq r$ and $1\leq i\leq m-1$. Now we have $I(v,z_j)$ for all $v\in a'_1\overline{w}$ and $1\leq j\leq r$, and so $z_1,\ldots,z_r$ consists of at most $n-1$ distinct lines. By choice of $r$, there are $1\leq j_1<\ldots<j_{m+1}\leq r$ such that $z_{j_s}=z_{j_t}=:z$ for all $1\leq s<t\leq m+1$. After re-indexing, we may assume $j_t=t$ for $1\leq t\leq m+1$. We have $I(b_{i_t},z)$ for all $1\leq t\leq m$, and so $z\in\acl(b_{i_1},\ldots,b_{i_m})\backslash\acl(\overline{c})$. Moreover, $I(b_{i_{m+1}},z)$ holds and $b_{i_{m+1}}\in\acl(b_{i_{m+1}})\backslash\acl(\overline{c})$. This contradicts $b_{i_{m+1}}\iind_{\overline{c}}b_{<i_{m+1}}$. 

Now let $b^0_{<\omega}\overline{c}^0=b_{<\omega}\overline{c}$ and construct $(b^j_{<\omega}\overline{c}^j)_{j\in\omega}$ such that, for all $j\in\omega$, $b^j_{<\omega}\overline{c}^j\equiv_\emptyset b_{<\omega}\overline{c}$ and $b^j_{<\omega}\overline{c}^j\trt{a}_\emptyset b^{<j}_{<\omega}\overline{c}^{<j}$. Then, for any $\sigma\colon \omega\to\omega$, we claim that $\{\varphi(x_1,x_2;b^j_{\sigma(j)},\overline{c}^j) \mid j\in\omega\}$ is consistent. While this is not hard to show directly, it follows from Proposition \ref{prop:Kequiv} since $a_1a_2\iind_\emptyset b\overline{c}$ and $(b^j_{\sigma(j)}\overline{c}^j)_{j\in\omega}$ is an $\trt{a}$-independent sequence in $\tp(b\overline{c})$. 

Finally, given $i,j\in\omega$, let $\overline{c}^j_i=\overline{c}^j$. By the last two paragraphs, the array $\{b^j_i\overline{c}^j_i \mid i,j\in\omega\}$ witnesses TP$_2$ for $\varphi(x_1,x_2;y_1,\ldots,y_n)$. 
\end{proof}

\begin{remark}\label{rem:KiPi}
The Kim-Pillay characterization of simple theories \cite{kim-pillay} says that a theory $T$ is simple if and only if there is an invariant ternary relation $\gind$ on $\M\models T$ satisfying monotonicity, base monotonicity, symmetry, transitivity, finite character, local character, full existence, and the independence theorem. It is therefore interesting to observe that in $T_{m,n}$, $\iind$ satisfies all axioms on this list, except for base monotonicity when $m,n\geq 2$. To verify this, it remains to show transitivity (which will be used in Lemma \ref{lem:mixed})  and local character (note that finite character is immediate from Lemma \ref{lem:SFCW}).

\emph{Transitivity}: Fix $D\seq C\seq B$ and suppose $A\iind_D C$ and $A\iind_C B$. Suppose we have $x\in\acl(AD)$ and $y\in\acl(B)$ such that $x$ and $y$ are either incident or equal. We want to show $x\in\acl(D)$ or $y\in\acl(D)$. Since $D\seq C$ and $A\iind_C B$, we have $x\in\acl(C)$ or $y\in\acl(C)$. If $x\in \acl(C)$ then $x\in\acl(AD)\cap\acl(C)$ and so $x\in \acl(D)$ since $A\iind_D C$. If $y\in\acl(C)$ then $A\iind_D C$ implies $x\in\acl(D)$ or $y\in\acl(D)$.

\emph{Local character}: We will use the following fact, which is evident from Adler's proof of local character for $\trt{a}$ (see \cite[Proposition 1.5]{adler}). For any $A$, $B$, and $C_0\seq B$ with $|C_0|\leq|A|+\aleph_0$, there is $C\seq B$ such that $|C|\leq|A|+\aleph_0$, $C_0\seq C$, and $A\trt{a}_C B$.

Fix $A$ and $B$, and let $\kappa=|A|+\aleph_0$. We construct $C_i\seq D_i\seq C_{i+1}\seq B$, for $i\in\omega$, such that $|C_i|\leq\kappa$, $A\trt{a}_{C_i} B$ and, if $x\in\acl(AC_i)$ and $y\in \acl(B)$ are incident, then $x\in \acl(D_i)$ or $y\in \acl(D_i)$. The construction is by induction where we assume we have $C_i$, then construct $D_i$ and $C_{i+1}$. For the base case, we start with $C_0\seq B$ so that $|C_0|\leq\kappa$ and $A\trt{a}_{C_0} B$. Suppose we have $C_i$ as desired. We claim that if $x\in\acl(AC_i)\backslash \acl(C_i)$, then there are at most $\max\{m-1,n-1\}$ elements $y\in \acl(B)$ incident to $x$. Indeed, otherwise we would have $x\in\acl(AC_i)\cap \acl(B)=\acl(C_i)$ (since $A\trt{a}_{C_i} B$). So let $E_i\seq \acl(B)$ be the set of elements in $\acl(B)$ incident to some element of $\acl(AC_i)\backslash \acl(C_i)$. By the claim, $|E_i|\leq|\acl(AC_i)|+\aleph_0\leq\kappa$. By finite character of algebraic closure, we may find $D_i\seq B$ such that $C_iE_i\seq\acl(D_i)$ and $|D_i|\leq\kappa$. If $x\in\acl(AC_i)$ and $y\in \acl(B)$ are incident, then either $x\in \acl(C_i)\seq \acl(D_i)$ or, by construction, $y\in E_i\seq \acl(D_i)$. Now choose  $C_{i+1}\seq B$ such that $|C_{i+1}|\leq\kappa$, $D_i\seq C_{i+1}$, and $A\trt{a}_{C_{i+1}}B$. 

Finally, let $C=\bigcup_{i\in\omega}C_i=\bigcup_{i\in\omega}D_i\seq B$. By construction $|C|\leq\kappa$. We show $A\iind_C B$. Fix $x\in\acl(AC)$ and $y\in \acl(B)$ such that $x$ and $y$ are either incident or equal. We show $x\in \acl(C)$ or $y\in \acl(C)$. By construction and finite character of algebraic closure, there is some $i\in\omega$ such that $x\in\acl(AC_i)$. If $x=y$ then $x\in \acl(C_i)\seq \acl(C)$ since $A\trt{a}_{C_i}B$. If $x$ and $y$ are incident, and $x\not\in \acl(C)$, then $x\in\acl(AC_i)\backslash \acl(C_i)$, and so $y\in E_i\seq \acl(D_i)\seq \acl(C)$. 
\end{remark}

\subsection{Forking independence and ``free independence''}\label{subsec:forking}

The goal of this section is to show $\trt{d}$ and $\trt{f}$ coincide in $T_{m,n}$ (i.e., forking and dividing are the same for complete types). To do this, we use a common strategy, also used in \cite{conant}, \cite{conant-terry}, and \cite{KrRa}, for showing that dividing in generic structures satisfies the extension axiom (from which it follows that $\trt{d}$ and $\trt{f}$ coincide). In particular, we define an auxiliary independence relation, which satisfies full existence and a ``mixed transitivity" property for dividing (see Lemma \ref{lem:mixed}). From this, extension for dividing independence follows immediately (see Corollary \ref{cor:f=d}). For $T_{m,n}$, this new relation is defined using the free completion of models of $T^p_{m,n}$.

\begin{definition}
Define $\ter{\otimes}$ on $\M\models T_{m,n}$ such that, given $A,B,C\subset\M$,
\[
\textstyle A\ter{\otimes}_C B~\Leftrightarrow ~A\iind_C B\text{ and }\acl(ABC)\cong_{ABC} F(\acl(AC)\acl(BC)).
\]
\end{definition}

We first prove a technical lemma which will be at the heart of several results involving $\ter{\otimes}$.

\begin{lemma}\label{lem:extF}
Suppose $B\models T^p_{m,n}$ and $A\seq B$ is $I$-closed in $B$. Then there is a substructure $C\seq F(B)$ such that $C\cong_A F(A)$, $A= B\cap C$, and there are no incidences between $C\backslash A$ and $B\backslash A$. 
\end{lemma}
\begin{proof}
Given $X,Y\seq F(B)$, let $N_X(Y)$ be the set of elements in $X$ incident to all elements in $Y$. We may write $F(B)$ as an increasing union $\bigcup_{k\in \omega}X_k$ such that $X_0=B$ and, for all $k\in \omega$, if $\Sigma_k$ is the collection of $m$-element sets $\sigma$ of points in $X_k$ such that $|N_{X_k}(\sigma)|\leq n-2$ and $n$-element sets $\sigma$ of lines in $X_k$ such that $|N_{X_k}(\sigma)|\leq m-2$, then there is a bijection $f_k\colon \Sigma_k\to X_{k+1}\backslash X_k$ such that $\sigma=N_{X_{k+1}}(\{f_k(\sigma)\})$ for all $\sigma\in \Sigma_k$. Define a sequence of subsets $Y_k\seq X_k$ so that $Y_0=A$ and
\[
Y_{k+1}=Y_k\cup \{f_k(\sigma) \mid \sigma\in\Sigma_k,~\sigma\seq Y_k\}.
\]
By induction, we show that $Y_k$ is $I$-closed in $X_k$. The base case $k=0$ is by the assumption that $A$ is $I$-closed in $B$. So assume $Y_k$ is $I$-closed in $X_k$ and suppose $\sigma\seq Y_{k+1}$ is an $m$-element set of points. We want to show $N_{X_{k+1}}(\sigma)\seq Y_{k+1}$. Suppose first that $\sigma\seq Y_k$. Then $N_{X_{k+1}}(\sigma)$ is contained in $N_{X_k}(\sigma)$ together with at most one point of the form $f_k(\sigma)\in X_{k+1}\backslash X_k$. Since $N_{X_k}(\sigma)\seq Y_k$ by induction, and $f_k(\sigma)\in Y_{k+1}$, we have $N_{X_{k+1}}(\sigma)\seq Y_{k+1}$. Now suppose $\sigma$ contains at least one element of $Y_{k+1}\backslash Y_k$. Hence there is $\tau\in\Sigma_k$ such that $\tau\seq Y_k$ and $f_k(\tau)\in\sigma$. Then $N_{X_{k+1}}(\sigma)\seq N_{X_{k+1}}(\{f_k(\tau)\})=\tau\seq Y_k\seq Y_{k+1}$.  By a symmetric argument, we conclude that $Y_{k+1}$ is $I$-closed in $X_{k+1}$.

For $k<\omega$, let $\Sigma^*_k=\{\sigma\in\Sigma_k \mid \sigma\seq Y_k\}$. Then $f_k\colon \Sigma^*_k\to Y_{k+1}\backslash Y_k$ is a bijection. Moreover, since $Y_k$ is $I$-closed in $X_k$, $\Sigma^*_k$ is precisely the collection of $m$-element sets $\sigma$ of points in $Y_k$ such that $|N_{Y_k}(\sigma)|\leq n-2$ and $n$-element sets $\sigma$ of lines in $Y_k$ such that $|N_{Y_k}(\sigma)|\leq m-2$. For any $\sigma\in\Sigma^*_k$, we have $N_{Y_{k+1}}(f_k(\sigma))\seq N_{X_{k+1}}(f_k(\sigma))=\sigma\seq Y_k$, and so $\sigma=N_{Y_{k+1}}(f_k(\sigma))$. Altogether, it follows that $F(A)$ is isomorphic to $C:=\bigcup_{k \in \omega}Y_k\seq F(B)$ (via an isomorphism extending inclusion on $A$). By construction, we have $A=B\cap C$. To finish the proof, we fix $u\in C\backslash A$ and $v\in B$, such that $u$ and $v$ are incident, and we show $v\in A$. By construction, $u=f_k(\sigma)$ for some $k\in\omega$ and $\sigma\in\Sigma_k$. Since $v\in B\seq X_{k+1}$ and $N_{X_{k+1}}(\{u\}) = \sigma$, it follows that $v\in \sigma\seq C$ and so $v\in B\cap C=A$. 
\end{proof}

The next proposition lists some properties of $\ter{\otimes}$. Other than full existence, these properties will not be crucial for the main results of this section, and so we only sketch the proof. 

\begin{proposition}\label{prop:stationary}$~$
\begin{enumerate}[$(a)$]
\item In $T_{m,n}$, $\ter{\otimes}$ satisfies monotonicity, transitivity, full existence, and stationarity.
\item In $T_{m,n}$, $\ter{\otimes}\Rightarrow\trt{f}\Rightarrow\trt{d}\Rightarrow\iind\Rightarrow\trt{a}$.
\end{enumerate}
\end{proposition}
\begin{proof}
Part $(a)$. Full existence for $\ter{\otimes}$ can be proven in the same way as full existence for $\iind$ in Remark \ref{rem:FE}, using the proof of Proposition~\ref{prop:amalgamation}, Corollary~\ref{cor:free-closure}, Proposition~\ref{prop:extend}, and Corollary~\ref{cor:extend}. Since $A\ter{\otimes}_C B$, together with the isomorphism types of $\acl(AC)$ and $\acl(BC)$ over $\acl(C)$, completely determines the isomorphism type of $\acl(ABC)$, stationarity follows from Corollary \ref{cor:extend}. Monotonicity and transitivity are  left to the reader (the main tool for the arguments is Lemma \ref{lem:extF}).

Part $(b)$. The only new implication is $\ter{\otimes}\Rightarrow\trt{f}$. In general, any invariant ternary relation satisfying the axioms from part $(a)$ implies $\trt{f}$ (see \cite[Theorem 4.1]{conant-terry}). 
\end{proof}

When $m,n\geq 2$, $T_{m,n}$ is not simple and so $\trt{d}$ must fail transitivity. However, we have the following ``mixed transitivity" statement involving both $\trt{d}$ and $\ter{\otimes}$.

\begin{lemma}\label{lem:mixed}
Suppose $A,B,C,D\subset\M\models T_{m,n}$ with $D\seq C\seq B$. Then
\[
\textstyle A\trt{d}_D C\text{ and }A\ter{\otimes}_C B~\Rightarrow~A\trt{d}_D B.
\]
\end{lemma}
\begin{proof}
We may  assume $B,C,D$ are algebraically closed. To show $A\trt{d}_D B$, it suffices by Corollary \ref{cor:div} to fix $E$ such that $D\seq E\seq B$, and show $A\iind_E B$. We may assume $E$ is algebraically closed. We first analyze $\acl(AE)$.

Let $E_0=C\cap E$. We claim that $X:=\acl(AE_0)E$ is $I$-closed in $Y:=\acl(AC)B$. Indeed, suppose $\sigma$ is a set of $m$ points in $X$ such that there is a line $y\in Y$ incident to every point in $\sigma$.  We have $A\trt{d}_D C$ and $D\seq E_0\seq C$, so Corollary \ref{cor:div} implies $A\iind_{E_0} C$. We also have $A\iind_C B$ by definition of $\ter{\otimes}$. By transitivity for $\iind$ (Remark \ref{rem:KiPi}), it follows that $A\iind_{E_0} B$. For a contradiction, suppose $y\not\in X$. Since $\acl(AE_0)$ and $E$ are each algebraically closed, it follows there are points $x_1,x_2\in\sigma$ such that $x_1\in \acl(AE_0)\backslash E_0$ and $x_2\in E\backslash E_0$ (note $\acl(AE_0)\cap E=E_0$ since $A\iind_{E_0}B$). Since $x_1\in\acl(AE_0)\backslash E_0$, $y\not\in E_0$, $I(x_1,y)$ holds, and $A\iind_{E_0}B$, it follows that $y\not\in B$. So $y\in\acl(AC)$. Then $I(x_2,y)$, $A\iind_C B$, and $x_2\in E\backslash E_0$  imply $y\in C$. But then $I(x_1,y)$, $A\iind_{E_0} C$, and $x_1\not\in E_0$ imply $y\in E_0\seq X$, which is a contradiction. By a symmetric argument, we see that $X$ is $I$-closed in $Y$.

Since $A\ter{\otimes}_C B$, we have $\acl(ABC)\cong F(Y)$ over $ABC$. So by Lemma \ref{lem:extF}, $\acl(X)\cong F(X)$, $\acl(X)\cap Y=X$, and there are no incidences between $\acl(X)\backslash X$ and $Y\backslash X$. Using this, we show $A\iind_E B$. Suppose $u\in\acl(AE)$ and $v\in B$ are either incident or equal. We want to show $u\in E$ or $v\in E$. Suppose first that $u\in X$. Then we may assume $u\in\acl(AE_0)$. Since $A\iind_{E_0} B$ we either have $u\in E_0\seq E$ or $v\in E_0\seq E$. Now suppose $u\in\acl(X)\backslash X$. Since $v\in B\seq Y$, we have $v\in X$, and we may again assume $v\in\acl(AE_0)$. So $v\in\acl(AE_0)\cap B=E_0\seq E$. 
 \end{proof}

\begin{corollary}\label{cor:f=d}
For any $A,B,C\subset\M\models T_{m,n}$, $A\trt{f}_C B$ if and only if $A\trt{d}_C B$ if and only if for all $D$ such that $C\seq D\seq\acl(BC)$, $A\iind_D B$.
\end{corollary}
\begin{proof}
The second equivalence just repeats Corollary \ref{cor:div}. So we need to show $\trt{d}\Rightarrow \trt{f}$. Fix $A,B,C$ such that $A\trt{d}_C B$, and suppose $\widehat{B}\supseteq BC$. By \emph{full existence} for $\ter{\otimes}$, there is $A'\equiv_{BC} A$ such that $A'\ter{\otimes}_{BC}\widehat{B}$. Since $A'\trt{d}_C BC$, we have $A'\trt{d}_C\widehat{B}$ by Lemma \ref{lem:mixed}.
\end{proof}

Corollary~\ref{cor:f=d} establishes that forking and dividing in $T_{m,n}$ coincide at the level of complete types. We now observe that they differ at the level of formulas whenever $T_{m,n}$ is not simple. 

\begin{proposition}\label{prop:forkingdividing}
For $m,n\geq 2$, there is a formula which forks but does not divide in $T_{m,n}$.
\end{proposition}
\begin{proof}
Let $\varphi(x_1,x_2;y_1,\dots,y_n)$ be the formula from Proposition~\ref{prop:TP2}, and let $a_1,a_2,b,\overline{c}\in \mathbb{M}$ be such that $\varphi(a_1,a_2;b,\overline{c})$ holds. It follows from the proof of that proposition that $\varphi(x_1,x_2;b,\overline{c})$ divides. Note also that the formulas $x=b$ and $x=c_1$ each divide. So the formula $\theta(x_1,x_2;b,\overline{c}):=(\varphi(x_1,x_2;b,\overline{c})\vee (x_1=b \lor x_1=c_1))$ forks. 

We claim that $\theta(x_1,x_2;b,\overline{c})$ does not divide. Indeed, let $(b^i,\overline{c}^i)_{i\in\omega}$ be an indiscernible sequence with $b^0 = b$ and $\overline{c}^0 = \overline{c}$. Let $\overline{d}$ be the subtuple of ${b}\overline{c}$ which is constant along the indiscernible sequence, so the sequence $(b^i,\overline{c}^i)_{i\in\omega}$ is $\trt{a}$-independent over $\overline{d}$. If $\overline{d}$ contains $b$ or $c_1$ then $\{(x_1=b \lor x_1=c_1)\mid i\in\omega\}$ is consistent, since either $b^i=b^0$ for all $i\in\omega$ or $c^i_1=c^0_1$ for all $i\in\omega$. So we may assume $\overline{d}$ does not contain $b$ or $c_1$. In this case, we note that $a_1a_2\iind_{\overline{d}}b\overline{c}$ (since both $a_1a_2\overline{d}$ and $b\overline{c}$ are algebraically closed, intersect in $\overline{d}$, and have no incidences between), and $\overline{d}$ is algebraically closed, so $\{\varphi(x_1,x_2;b^i,\overline{c}^i)\mid i\in \omega\}$ is consistent by Proposition~\ref{prop:Kequiv}.
\end{proof}

The result of Proposition~\ref{prop:forkingdividing} can be easily adjusted to take place over any set of parameters, instead of over the empty set. The equivalence of forking and dividing for complete types, but not for formulas, is yet another phenomenon that $T_{m,n}$ shares with the Henson graphs \cite{conant} and many NSOP$_1$ examples from \cite{KrRa}. On the other hand, there is an interesting difference between the present situation and that of the Henson graphs (which are not NSOP$_1$). Specifically, the forking and non-dividing formula in Proposition \ref{prop:forkingdividing} is obtained as the disjunction of a dividing formula and an \emph{algebraic} dividing formula. This cannot happen in the Henson graphs.

\subsection{Elimination of imaginaries and thorn-forking}\label{subsec:ei}

Recall that a theory $T$ has \emph{elimination of imaginaries} if for every imaginary $e\in \mathbb{M}^{\eq}$, there exists a real tuple $a\in \mathbb{M}$ such that $e\in \dcl^{\eq}(a)$ and $a\in \dcl^{\eq}(e)$. $T$ has \emph{weak elimination of imaginaries} if for every imaginary $e\in \mathbb{M}^{\eq}\models T^{\eq}$, there exists a real tuple $a\in \mathbb{M}$ such that $e\in \dcl^{\eq}(a)$ and $a\in \acl^{\eq}(e)$.

\begin{remark}
For all $m$ and $n$, $T_{m,n}$ fails to code pairs, and hence does not have elimination of imaginaries. Let $\sim$ be the definable equivalence relation on $\mathbb{M}^2$ given by $(a,b)\sim (a',b')$ if and only if $\{a,b\} = \{a',b'\}$. 

First, assume we are not in the case $m = n = 2$. Then, by symmetry, we may assume $m\geq n$, so we may assume $m\geq 3$, or $m \leq 2$ and $n = 1$. In any of these cases, for any pair of distinct points $a$ and $b$, $\acl(ab) = \dcl(ab) = \{a,b\}$, so $\tp(ab) = \tp(ba)$ by Corollary~\ref{cor:acltype}. Thus there is an automorphism $\sigma$ of $\mathbb{M}$ swapping $a$ and $b$, and $\sigma$ fixes the imaginary $e = [(a,b)]_{\sim}$. But any real tuple definable from $ab$ must be a subtuple of $(a,b)$, and only the empty subtuple is fixed by $\sigma$. The empty tuple is not interdefinable with $e$, since $e$ is not fixed by all automorphisms of $\mathbb{M}$. So $T_{m,n}$ does not eliminate $\sim$.

Now consider the case $m = n = 2$. Let $a$ and $b$ be distinct points, with connecting line $\br$. Then $\acl(ab) = \dcl(ab) = \{a,b,\br\}$. Again, there is an automorphism $\sigma$ of $\mathbb{M}$ swapping $a$ and $b$ and fixing the imaginary $e = [(a,b)]_{\sim}$. Of course, $\sigma$ also fixes $\br$. A real tuple interdefinable with $e$ must be a subtuple of $(a,b,\br)$  fixed by $\sigma$, and $e$ is not interdefinable with the empty tuple, so the only option is the element $\br$. But now let $\{c,d\}$ be another pair of points on $\br$, distinct from $\{a,b\}$. By Corollary \ref{cor:acltype}, there is an automorphism of $\mathbb{M}$ which fixes $\br$ and moves $(a,b)$ to $(c,d)$ (hence does not fix $e$). So $T_{2,2}$ also fails to eliminate $\sim$.
\end{remark}

We now show $T_{m,n}$ has weak elimination of imaginaries, using a criterion observed by Montenegro and Rideau in \cite{MoRi} (we thank them for allowing us to include the result). This criterion was abstracted from a method used originally by Hrushovski~\cite{hrushovski}, and then in several other contexts (see~\cite{ChaPil} and~\cite{KrRa}), in order to deduce weak elimination of imaginaries from an independence theorem.

\begin{proposition}\label{prop:silvain}\textnormal{\cite[Proposition 1.17]{MoRi}}
Suppose there is a ternary relation $\gind$ on $\mathbb{M}\models T$, satisfying the following properties:
\begin{enumerate}[$(i)$]
\item Given $a,b\in \mathbb{M}$ and $C^* = \acl^{\eq}(C^*) \subset \mathbb{M}^{\eq}$, and letting $C = C^*\cap \mathbb{M}$, there exists $a'\equiv_{C^*} a$ such that $a'\gind_C b$.
\item Given $a,b,c\in \mathbb{M}$ and $C=\acl(C)\subset \mathbb{M}$ such that $a\equiv_C b$, $b\gind_C a$, and $c\gind_C a$, there exists $c'$ such that $c'a \equiv_C c'b\equiv_C ca$.
\end{enumerate}
Then $T$ has weak elimination of imaginaries.
\end{proposition}
\begin{proof}
Fix $e\in \mathbb{M}^{\eq}$, and let $C^* = \acl^{\eq}(e)$ and $C = C^*\cap \mathbb{M}$. It suffices to show that $e\in \dcl^{\eq}(C)$, so pick some $\sigma\in \Aut(\mathbb{M}^{\eq}/C)$. We will show that $\sigma$ fixes $e$. 

Pick $a\in\M$ and a $\emptyset$-definable function $f$ such that $f(a)=e$. By $(i)$, there exist $b\equiv_{C^*}\sigma(a)$ and $c\equiv_{C^*} a$ such that $b\gind_C a$ and $c\gind_C a$. Note that $f(c)=e$. Since $a\equiv_C \sigma(a)\equiv_C b$, we apply $(ii)$ to find $c'$ such that $c'a\equiv_C c'b\equiv_C ca$. Now $f(a)=f(c)$ implies $f(a)=f(c')$ and $f(b)=f(c')$. So $f(b)=e$, which implies $f(\sigma(a))=e$, and thus $\sigma(e)=e$.
\end{proof}

\begin{remark}\label{rem:realalg}
Let $C^*\subseteq \mathbb{M}^{\eq}\models T^{\eq}$, and let $C = C^* \cap \mathbb{M}$. For any $a,b\in \mathbb{M}$, if $a\trt{a}_{C^*} b$ in $\mathbb{M}^{\eq}$ then $a\trt{a}_C b$ in $\mathbb{M}$. Indeed, suppose we have $c\in \mathbb{M}$ such that $c\in \acl(Ca)\cap \acl(Cb)$. Then also $c\in \acl^{\eq}(C^*a) \cap \acl^{\eq}(C^*b)$, so $c\in C^* \cap \mathbb{M} = C$.
\end{remark}

\begin{theorem}\label{thm:WEI}
For all $m$ and $n$, $T_{m,n}$ has weak elimination of imaginaries. 
\end{theorem}
\begin{proof}
We apply Proposition~\ref{prop:silvain} to our independence relation $\iind$. Since condition $(ii)$ is a special case of the independence theorem over algebraically closed sets (and so follows directly from Lemma~\ref{lem:IT}), it suffices to show that $\iind$ satisfies condition $(i)$. So suppose we have $a,b\in \mathbb{M}$, $C^* = \acl^{\eq}(C^*)\subset \mathbb{M}^{\eq}$, and $C = C^*\cap \mathbb{M}$. We want to find $a'\equiv_{C^*}a$ such that $a'\iind_C b$. In the following argument, when we say that a sequence of tuples in some type over $C^*$ is $\trt{a}$-independent, we mean algebraic independence in the sense of $\M^{\eq}$. 

Let $\kappa$ be a cardinal which is larger than the number of $\cL_C$-formulas. By full existence for $\trt{a}$ (in $\M^{\eq})$, we may let $\overline{b} = (b_\alpha)_{\alpha< \kappa}$ be an $\trt{a}$-independent sequence in $\tp(b/C^*)$, and then let $(\overline{b}\!\,^\beta)_{\beta<\kappa}$ be an $\trt{a}$-independent sequence in $\tp(\overline{b}/C^*)$. By Remark \ref{rem:realalg}, the sequences $(\overline{b}\!\,^\beta)_{\beta<\kappa}$ and $(b^\beta_\alpha)_{\alpha<\kappa}$, for each $\beta<\kappa$, are $\trt{a}$-independent over $C$ in the sense of $\M$. 

Suppose for contradiction that there is no $a'\equiv_{C^*} a$ such that $a'\iind_C b$. Fix $\beta<\kappa$. For any $\alpha<\kappa$, picking an automorphism $\sigma\in \Aut(\mathbb{M}^{\eq}/C^*)$ with $\sigma(b^\beta_\alpha) = b$, we have $\sigma(a) \niind_C b$, so $a\niind_C b^\beta_\alpha$. Lemma~\ref{lem:SFCW} tells us that this failure of independence is witnessed by some $\cL_C$-formula, and by pigeonhole, we may refine $\overline{b}\!\,^\beta$ to an infinite subsequence $\overline{c}^\beta = (b^\beta_{\alpha_i})_{i\in\omega}$ and assume that the same formula $\varphi(x;y)$ witnesses $a\niind_C c^\beta_i$ for all $i\in \omega$. By Lemma~\ref{lem:SFCW}$(iii)$, we have $a\ntrt{a}_C \overline{c}^\beta$. 

Having constructed $\overline{c}^\beta$ for all $\beta<\kappa$, we can refine to a subsequence $(\overline{c}^{\beta_i})_{i\in \omega}$ such that the same $\cL_C$-formula $\psi(x,y)$ witnesses $a\ntrt{a}_C \overline{c}^{\beta_i}$ for all $i\in \omega$ in the sense of Fact~\ref{fact:SFCWalg}. This contradicts Fact~\ref{fact:SFCWalg}$(iii)$, since $(\overline{c}^{\beta_i})_{i\in \omega}$ is $\trt{a}$-independent over $C$.
\end{proof}

Weak elimination of imaginaries for $T_{m,n}$ allows us to analyze thorn-forking. Recall that M-\emph{independence} $\trt{M}$ is defined by forcing base monotonicity on $\trt{a}$, and  \emph{thorn independence} $\trt{\thorn}$ is defined by forcing extension on $\trt{M}$ (see \cite{adler} for details). A complete theory $T$ is called \emph{rosy} if $\trt{\thorn}$ is symmetric \emph{as a relation constructed in $T^{\eq}$}.  If $T$ has weak elimination of imaginaries, then it suffices to consider thorn-forking in $T$. Therefore, in the following discussion we work with $\trt{M}$ and $\trt{\thorn}$ as constructed in $T$.

Before considering thorn independence in $T_{m,n}$, let us recall the situation with some related theories which also have weak elimination of imaginaries. First, in the theory $T_n$ of the generic $K_n$-free graph, algebraic independence satisfies base monotonicity, and thus $\trt{a}=\trt{\thorn}$. So $T_n$ is rosy, and thorn independence is fairly uninteresting since algebraic closure is trivial. Forking independence in $T_n$ is a stronger and more complicated notion (see \cite{conant}). 

Now, given $k\geq 1$, let $T^k$ be the the theory of the ``generic $k$-ary function", i.e., the model completion of the empty theory in a language consisting of a single $k$-ary function symbol. In \cite{KrRa}, $T^k$ is shown to be NSOP$_1$ with weak elimination of imaginaries, and not simple when $k\geq 2$ (arbitrary languages are considered in \cite{KrRa}, but we focus on $T^k$ for simplicity). Like in $T_{m,n}$, Kim independence in $T^k$ can be extended to arbitrary sets, and forking independence, which coincides with dividing independence, is obtained by forcing base monotonicity on Kim independence. Unlike in $T_{m,n}$, Kim independence in $T^k$ coincides with algebraic independence. So $\trt{f}=\trt{\thrn}=\trt{M}$ in $T^k$, and $T^k$ is not rosy when $k\geq 2$. In $T_{m,n}$, a similar situation occurs. 

\begin{proposition}
Given $A,B,C\subset\M\models T_{m,n}$, $A\trt{f}_C B$ if and only if $A\trt{\thorn}_C B$. Thus $T_{m,n}$ is not rosy when $m,n\geq 2$.
\end{proposition}
\begin{proof}
The second claim follows from the first, since $T_{m,n}$ has weak elimination of imaginaries, and is not simple when $m,n\geq 2$.  To show the first claim, recall that $\trt{f}\Rightarrow\trt{\thorn}$ and $\trt{\thorn}$ satisfies base monotonicity (in any theory), so it is enough to show $\trt{\thorn}\Rightarrow\iind$ by Corollary \ref{cor:f=d}. So suppose $A\trt{\thorn}_C B$. Since $\trt{\thorn}$ satisfies extension, we may assume $C\seq B=\acl(B)$. Since $\trt{\thorn}\Rightarrow\trt{a}$, it suffices to fix $u\in\acl(AC)$ and $v\in B$, such that $u$ and $v$ are incident, and show $u\in\acl(C)$ or $v\in\acl(C)$. Without loss of generality, assume $I(u,v)$ (the other case is symmetric). Suppose $v\not\in\acl(C)$. Then $v$ is incident to $t\leq m-1$ points in $\acl(C)$.  Extend the structure $\acl(C)v$ by $m-t-1$ new points incident only to $v$, and embed the free completion into $\M$ over $\acl(C)v$. Let $\overline{w}=(w_1,\ldots,w_{m-t-1})$ be the image of the new points. Then $v$ is incident to $m-1$ distinct points in $\acl(C\overline{w})$, but $v\not\in\acl(C\overline{w})$ by Lemma \ref{lem:extF}.  Now, by extension for $\trt{\thorn}$, there is $A'\equiv_B A$ such that $A'\trt{\thorn}_C B\overline{w}$. Let $u'\in \acl(A'C)$ be such that $u'\equiv_B u$. Then $I(u',v)$ holds. If $u'\not\in\acl(C\overline{w})$, then there are $m$ distinct points in $C\overline{w}u'$ incident to $v$, so $v\in\acl(C\overline{w}u')$. It follows that $v\in\acl(C\overline{w})$, since base monotonicity for $\trt{\thorn}$ implies  $u'\trt{a}_{C\overline{w}}B\overline{w}$. This contradicts  $v\not\in\acl(C\overline{w})$, so we must have $u'\in\acl(C\overline{w})$, which implies $u'\in\acl(C)$ since $u'\trt{a}_C \overline{w}$. Since $u'\equiv_B u$, we have $u\in\acl(C)$, as desired.
\end{proof}

On the other hand, if $m,n\geq 2$, then M-independence fails extension in $T_{m,n}$. Indeed, since $n\geq 2$ we may fix $m$ points $\overline{a}=(a_1,\ldots,a_m)$ and a line $b$ such that $I(a_i,b)$ for all $1\leq i\leq m$. Since $m,n\geq 2$, we have $a_m\trt{M}_\emptyset b$. Let $C=\{a_1,\ldots,a_{m-1}\}$. For any $a'_m\equiv_b a_m$, we have $a'_m\ntrt{a}_C b$, and so $a'_m\ntrt{M}_\emptyset Cb$. It follows that M-independence does not coincide with thorn independence, unlike in the case of $T^k$.

\subsection{The stable cases}\label{subsec:stable}

In this final section, we settle the remaining details concerning $T_{m,1}$ and $T_{1,n}$. Since $T_{1,n}$ is interdefinable with $T_{n,1}$, we focus on  $T_{m,1}$ for $m\geq 1$. While $T_{1,1}$ and $T_{2,1}$ are classical examples, and $T_{3,1}$ appears in \cite{saffe} (see Remark \ref{rem:DOP} below), we could not find a proof of the following fact in the literature, and so we include it for the sake of completeness.

\begin{theorem}\label{thm:stable}
$T_{m,1}$ is $\omega$-stable of Morley rank $m$.
\end{theorem}
\begin{proof}
 In $\M\models T_{m,1}$, algebraic closure is disintegrated and, if an incidence $I(a,b)$ holds, then $a\in\acl(b)$. Combined with Corollary \ref{cor:acltype}, it is possible to show $\omega$-stability by counting types; but here is another argument. The observations about algebraic closure imply $\ter{\otimes}=\iind=\ter{a}$, and so $\ter{\otimes}=\trt{f}=\trt{a}$ by Proposition \ref{prop:stationary}(b). In particular, nonforking is stationary by Proposition~\ref{prop:stationary}(a), which implies $T_{m,1}$ is stable. Since algebraic closure is locally finite and modular, it follows that $\trt{a}$ satisfies finitary local character, and so $T_{m,1}$ is superstable. Since $T_{m,1}$ is $\aleph_0$-categorical (see Remark \ref{rem:m=1}), we conclude that it is $\omega$-stable of finite Morley rank, which coincides with its $U$-rank (see \cite{CHL}).  It remains to show that this rank is precisely $m$. 

To see $U(T_{m,1})\geq m$, fix $a\in L(\M)$ and pairwise distinct $b_1,\ldots,b_{m-1}\in P(\M)$ such that $I(b_i,a)$ for all $1\leq i\leq m-1$. For $0\leq i\leq m-1$ set $B_i=\{b_j:j\leq i\}$ (so $B_0=\emptyset$), and set $B_m=B_{m-1}\cup\{a\}$. Then $a\ntrt{a}_{B_{i-1}}B_i$ for all $1\leq i\leq m$, and so $U(a/\emptyset)\geq m$. Conversely, fix $k\geq 1$ and assume $U(T_{m,1})\geq k$. We show $k\leq m$. Since algebraic closure is disintegrated, the assumption implies that there are $a\in\M$ and $\emptyset=B_0\seq B_1\seq\ldots\seq B_k\subset\M$ such that, for all $1\leq i\leq k$, there is $b_i\in(\acl(a)\cap\acl(B_i))\backslash \acl(B_{i-1})$.  Note that $a\neq b_i$ for all $1\leq i\leq k-1$ since, otherwise, we would have $a\in\acl(B_i)$ and thus $b_{i+1}\in\acl(a)\seq\acl(B_i)$, which is a contradiction. Moreover, $b_1,\ldots,b_{k-1}$ are pairwise distinct since $b_i\in\acl(B_i)\backslash\acl(B_{i-1})$ for all $1\leq i\leq k-1$. For each $1\leq i\leq k-1$, since $b_i\in\acl(a)$ and $a\neq b_i$, it follows that $b_i\in P(\M)$, $a\in L(\M)$, and $I(b_i,a)$. Altogether, $a$ is a line incident to $k-1$ distinct points, which means $k\leq m$.
\end{proof}

\begin{remark}\label{rem:DOP}
The theory $T_{1,1}$ is interdefinable with the theory of two unary predicates which partition the universe into two infinite pieces, and so $I(T_{1,1},\aleph_\alpha)=2|\alpha|+1$. The theory $T_{2,1}$ has quantifier elimination after adding a symbol for the \emph{equivalence relation} $E(x,y)$ on $L$ defined by $\exists z(I(z,x)\wedge I(z,y))$. In a model of $T_{2,1}$, each point determines a distinct $E$-class of lines, and so models of $T_{2,1}$ are bi-interpretable with models of the theory of an equivalence relation with infinitely many infinite classes. An easy counting exercise then shows $I(T_{2,1},\aleph_\alpha)=2^{|\alpha|}+\aleph_0$ when $\alpha>0$. In contrast, for $m\geq 3$, $T_{m,1}$ has the maximum number of models in all uncountable cardinalities. To see this, fix $\kappa>\aleph_0$ and an $(m-1)$-uniform hypergraph $G$ on $\kappa$. Then one can construct a model $M_G\models T_{m,1}$ such that $P(M_G)=\kappa$ and, for any $(m-1)$-element subset $X$ of $\kappa$, the set of lines in $M_G$ incident to all points in $X$ has size $\kappa$ if $X$ is an edge in $G$, and $\aleph_0$ otherwise. So the $2^\kappa$ pairwise non-isomorphic choices for $G$ produce $2^\kappa$ pairwise non-isomorphic models of $T_{m,1}$ of size $\kappa$. Coding graphs in models is an essential component of the proof that a superstable theory with the dimensional order property (DOP) has the maximum number of models in uncountable cardinalities. We leave it as an exercise to show that if $m\geq 3$, then $T_{m,1}$ has DOP. In fact, $T_{3,1}$ is interdefinable with the example given in \cite[Section 2.1]{saffe}.
\end{remark}


\begin{thebibliography}{1}

\bibitem{adler}
Hans Adler, \emph{A geometric introduction to forking and thorn-forking}, J.
  Math. Log. \textbf{9} (2009), no.~1, 1--20. \MR{2665779 (2011i:03026)}
  
\bibitem{baldwin}
John~T.\ Baldwin, \emph{An almost strongly minimal non-{D}esarguesian projective
  plane}, Trans. Amer. Math. Soc. \textbf{342} (1994), no.~2, 695--711.
  \MR{1165085}

 \bibitem{BaCa} Silvia Barbina and Enrique Casanovas, \emph{Model theory of Steiner triple systems}, arXiv:1805.06767, 2018.
  
\bibitem{casanovas} Enrique Casanovas, \emph{Simple theories and hyperimaginaries}, Lecture Notes in Logic, vol. 39, Cambridge University Press, 2011.


\bibitem{CK}
C.\ C.\ Chang and H.\ J.\ Keisler, \emph{Model theory}. Studies in Logic and the Foundations of Mathematics, Vol.\ 73. Third edition. North-Holland Publishing Co., Amsterdam, 1990.

  \bibitem{ChaPil}
Z.~Chatzidakis and A.~Pillay, \emph{Generic structures and simple theories},
  Ann. Pure Appl. Logic \textbf{95} (1998), no.~1-3, 71--92. \MR{1650667
  (2000c:03028)}
  
  \bibitem{CHL}
G.~Cherlin, L.~Harrington, and A.~H. Lachlan, \emph{{$\aleph_0$}-categorical,
  {$\aleph_0$}-stable structures}, Ann. Pure Appl. Logic \textbf{28} (1985),
  no.~2, 103--135. \MR{779159}

\bibitem{ArtemNick}
Artem Chernikov and Nicholas Ramsey, \emph{On model-theoretic tree properties},
  J. Math. Log. \textbf{16} (2016), no.~2, 1650009, 41. \MR{3580894}
  
  \bibitem{conant} Gabriel Conant, \emph{Forking and dividing in Henson graphs}, Notre Dame J. Formal Logic \textbf{58} (2017), no.~4, 555-566.
  
  \bibitem{conant-terry}
Gabriel Conant and Caroline Terry, \emph{Model theoretic properties of the
  {U}rysohn sphere}, Ann. Pure Appl. Logic \textbf{167} (2016), no.~1, 49--72.
  \MR{3413493}
  
  \bibitem{erdos}
Paul Erd{\H o}s, \emph{Some old and new problems in various branches of
  combinatorics}, Proceedings of the {T}enth {S}outheastern {C}onference on
  {C}ombinatorics, {G}raph {T}heory and {C}omputing ({F}lorida {A}tlantic
  {U}niv., {B}oca {R}aton, {F}la., 1979), Congress. Numer., XXIII--XXIV,
  Utilitas Math., Winnipeg, Man., 1979, pp.~19--37. \MR{561032}
  
  \bibitem{Hall}
  Marshall Hall, \emph{Projective planes}, Trans. Amer. Math. Soc. \textbf{54} (1943), 229--277. \MR{0008892}
  
  \bibitem{Henson}
  C. Ward Henson, \emph{A family of countable homogeneous graphs}, Pacific J. Math. Volume 38, Number 1 (1971), 69--83.
  
  \bibitem{HW}
  Joram Hirschfeld and William H. Wheeler, \emph{Forcing, Arithmetic, Division Rings}, Lecture Notes in Mathematics, vol. 454, Springer-Verlag, Berlin, 1975.
  
  \bibitem{hrushovski} 
  Ehud Hrushovski, \emph{Pseudofinite fields and related structures}. In \emph{Model theory and applications}, volume 11 of Quad. Mat., 
151--212. Aracne, Rome, 2002.

\bibitem{KRKim}
Itay Kaplan and Nicholas Ramsey, \emph{On {K}im-independence},
  arXiv:1702.03894, 2017.
  
\bibitem{Kegel}
Otto H. Kegel, \emph{Existentially closed projective planes}, 165--174, in \emph{Geometry --- von Staudt's Point of View}, Ed. Peter Plaumann and Karl Strambach, Springer Netherlands, Dordrecht (1981).

\bibitem{kim-pillay}
Byunghan Kim and Anand Pillay, \emph{Simple theories}, Ann. Pure Appl. Logic
  \textbf{88} (1997), no.~2-3, 149--164, Joint AILA-KGS Model Theory Meeting
  (Florence, 1995). \MR{1600895 (99b:03049)}
  
  \bibitem{Kr} 
  Alex Kruckman, \emph{Disjoint $n$-amalgamation and pseudofinite countably categorical theories}, arXiv:1510.03539, 2015.
  
  \bibitem{KrRa}
Alex Kruckman and Nicholas Ramsey, \emph{Generic expansion and {S}kolemization in NSOP$_1$ theories},
  Ann. Pure Appl. Logic \textbf{169} (2018), no.~8, 755--774.
  
  
   \bibitem{MoRi} Samaria Montenegro and Silvain Rideau, \emph{Imaginaries in pseudo $p$-adically closed fields}, arXiv:1802.00256, 2018.
  
  \bibitem{MoWi}
G.~Eric Moorhouse and Jason Williford, \emph{Embedding finite partial linear
  spaces in finite translation nets}, J. Geom. \textbf{91} (2009), no.~1-2,
  73--83. \MR{2471997}
  
  \bibitem{saffe}
J\"urgen Saffe, \emph{A superstable theory with the dimensional order property
  has many models}, Proceedings of the {H}erbrand symposium ({M}arseilles,
  1981), Stud. Logic Found. Math., vol. 107, North-Holland, Amsterdam, 1982,
  pp.~281--286. \MR{757035}

\bibitem{Sh500}
Saharon Shelah, \emph{Toward classifying unstable theories}, Ann. Pure Appl.
  Logic \textbf{80} (1996), no.~3, 229--255. \MR{1402297 (97e:03052)}
  
  \bibitem{TZ}
  K. Tent and B. Zilber, \emph{A non-desarguesian projective plane}, arXiv:1510.06176, 2012.
  
  
  \bibitem{Welsh}
D.~J.~A. Welsh, \emph{Matroid theory}, Academic Press [Harcourt Brace
  Jovanovich, Publishers], London-New York, 1976, L. M. S. Monographs, No. 8.
  \MR{0427112}

\end{thebibliography}
\end{document}